\documentclass[preprint,11pt]{elsarticle}

\usepackage{amssymb}
\usepackage{amsthm}

\newtheorem{thm}{Theorem}[section]
\newtheorem{cor}[thm]{Corollary}
\newtheorem{lem}[thm]{Lemma}
\newtheorem{prop}[thm]{Proposition}
\theoremstyle{definition}
\newtheorem{defn}[thm]{Definition}
\theoremstyle{definition}
\newtheorem{ex}[thm]{Example}
\theoremstyle{remark}
\newtheorem{rem}[thm]{Remark}
\usepackage{multicol}

\usepackage{fancyhdr}
\usepackage[fleqn]{amsmath}
\usepackage{amsthm}
\usepackage{layout}
\usepackage{setspace}
\usepackage{diagrams}
\usepackage{dsfont}
\usepackage{mathrsfs}
\usepackage{yfonts}
\usepackage{ifsym}
\usepackage{verbatim}
\usepackage{bbm}
\usepackage{stmaryrd}
\usepackage{txfonts}
\usepackage{mathtools}
\usepackage[all]{xy}
\usepackage{fullpage}

\journal{Journal of Pure and Applied Algebra}

\newcommand{\vvar}[1]{\overrightarrow{#1}}
\newcommand{\param}[1]{\hat{#1}}
\newcommand{\y}{\mathit{y}}
\newcommand{\op}{\mathrm{op}}
\newcommand{\tensor}{\bullet}
\newcommand{\sfun}[1]{\mathcal{F}_{#1}}
\newcommand{\Sfun}[1]{\mathcal{F}'_{#1}}

\newcommand{\To}{\rightarrow}
\newcommand{\Def}{:=}
\newcommand{\eq}{\equiv}
\newcommand{\e}[2]{{#1} \Rightarrow {#2}}

\newcommand{\Set}{\mathbf{Set}}
\newcommand{\Clone}{\mathbf{Clone}}

\newcommand{\N}{\mathbb{N}}
\newcommand{\ff}{\mathbb{F}}
\newcommand{\F}{\mathbb{L}}
\newcommand{\M}{\mathbb{M}} 			
\newcommand{\E}{\mathcal{E}}			
\newcommand{\Law}{\mathcal{L}}
\newcommand{\Mlaw}{\mathcal{M}}
\newcommand{\IL}{\textswab{E}}			
\newcommand{\trans}{\tauup}

\newcommand{\tnat}{\nu}
\newcommand{\tnatbar}{\bar{\nu}}
\newcommand{\nat}{\muup}
\newcommand{\natbar}{\bar{\muup}}

\newcommand{\term}{\mathrm{t}}

\newcommand{\SOAT}{\mathbf{SOAT}}
\newcommand{\SOEP}{\mathbf{SOEP}}
\newcommand{\LAW}{\mathbf{LAW}}		

\newcommand{\Mod}{\mathbf{Mod}}
\newcommand{\FMod}{\mathbf{FMod}}

\newcommand{\alg}[2]{\llbracket {#1} \rrbracket_{#2}}
\newcommand{\Alg}[1]{{#1}\mbox{-}\mathrm{\mathbf{Alg}}}

\newcommand{\FMOD}[1]{\mathbf{\mathbbm{Mod}}({#1})}

\newcommand{\cat}{\mathscr}

\newcommand{\mon}{\mathrm{\mathbf{T}}}

\newcommand{\emptycon}{-}

\newcommand{\tuple}[2]{( #1_1, \dots, #1_{#2} )}
\newcommand{\etuple}[3]{( #1_1 + #3, \dots, #1_{#2}+#3 )}
\newcommand{\card}[1]{\| #1 \|}
\newcommand{\vcon}[2]{#1_1, \dots, #1_{#2}} 	
\newcommand{\ivcon}[3]{#1_1^{(#3)}, \dots, #1_{#2_{#3}}^{(#3)}}
\newcommand{\mcon}[2]{\textsc{#1}_1 \colon [#1_1], \dots, \textsc{#1}_{#2} \colon [#1_{#2}]}	
\newcommand{\emcon}[3]{\textsc{#1}_1 \colon [#1_1 + #3], \dots, \textsc{#1}_{#2} \colon [#1_{#2} + #3]}	
\newcommand{\mvar}[1]{\textsc{#1}}
\newcommand{\Sep}{\triangleright}
\newcommand{\ev}{\mathbbm{e}}		
\newcommand{\lam}{\mathbbm{l}}	

\newcommand{\squad}{\hspace{4pt}}

\begin{document}

\begin{frontmatter}



\title{Functorial Semantics of Second-Order Algebraic Theories}

\author{Marcelo Fiore\footnote{Computer Laboratory, University of Cambridge} and Ola Mahmoud\footnote{Faculty of Mathematics and Statistics, University of St. Gallen}}

\begin{abstract}
The purpose of this work is to complete the algebraic foundations of second-order languages from the viewpoint of categorical algebra as developed by Lawvere. To this end, this paper introduces the notion of second-order algebraic theory and develops its basic theory.  A crucial role in the definition is played by the second-order theory of equality $\M$, representing the most elementary operators and equations present in every second-order language. The category $\M$ can be described abstractly via the universal property of being the free cartesian category on an exponentiable object. Thereby, in the tradition of categorical algebra, a second-order algebraic theory consists of a cartesian category $\Mlaw$ and a strict cartesian identity-on-objects functor $\M \to \Mlaw$ that preserves the universal exponentiable object of $\Mlaw$. Lawvere's functorial semantics for algebraic theories can then be generalised to the second-order setting. To verify the correctness of our theory, two categorical equivalences are established: at the syntactic level, that of second-order equational presentations and second-order algebraic theories; at the semantic level, that of second-order algebras and second-order functorial models.
\end{abstract}


\end{frontmatter}



\section{Introduction}
\label{introduction}

Algebra is the study of operations on mathematical structures, and the constructions and relationships arising from them. These structures span the most basic algebraic entities, such as arithmetic, to the more abstract, such as groups, rings, lattices, etc. Based on these, Birkhoff \cite{Birkhoff1935} laid out the foundations of a general unifying theory, now known as universal algebra. His formalisation of the notion of algebra starts with the introduction of equational presentations. These constitute the syntactic foundations of the subject. Algebras are then the semantics, or model theory, and play a crucial role in establishing the logical foundations. Indeed, Birkhoff introduced equational logic as a sound and complete formal deductive system for reasoning about algebraic structure.

The investigation of algebraic structure was further enriched by Lawvere's fundamental work on algebraic theories \cite{Lawvere2004}. His approach gives an elegant categorical framework for providing a presentation-independent treatment of universal algebra, and it embodies the motivation for the present work. 

As per Lawvere's own philosophy, we believe in the inevitability of algebraic content in mathematical subjects. We contend that it is only by looking at algebraic structure from all perspectives $-$ syntactic, semantic, categorical $-$ and the ways in which they interact, that the subject is properly understood. In the context of logic, algebra and theoretical computing, for instance, consider that: (i) initial-algebra semantics provides canonical compositional interpretations \cite{GoguenThatcherWagner1978}; (ii) free constructions amount to abstract syntax \cite{Mccarthy1963} that is amenable to proofs by structural induction and definitions by structural recursion \cite{Burstall1969}; (iii) equational presentations can be regarded as bidirectional rewriting theories and studied from a computational point of view \cite{KnuthBendix1970}; (iv) algebraic theories come with an associated notion of algebraic translation \cite{Lawvere2004}, whose syntactic counterpart provides the right notion of syntactic translation between equational presentations \cite{Fujiwara1959, Fujiwara1960}; (v) strong monads have an associated metalogic from which equational logics can be synthesised \cite{FioreHur2008b, FioreHur2010}. 

The realm of categorical universal algebra has so far been restricted to first-order languages. We further extend it to include languages with variable-binding, such as the $\lambda$-calculus \cite{Aczel1978} and predicate logic \cite{Aczel1980}. Emulating Lawvere's framework will enable us to:
\begin{itemize}
\item[-] define second-order algebraic theories to be structure preserving functors from a suitable base category, the second-order theory of equality, to a category which abstractly \emph{classifies} a given second-order presentation,
\item[-] extract syntactic information via internal languages from the categorical framework of second-order algebraic theories,
\item[-] synthesise a notion of syntactic translation from the canonical notion of morphism of algebraic theories, and vice versa,
\item[-] establish the functoriality of second-order semantics;
\end{itemize}
all in such a way that the expected categorical equivalences are respected. More precisely, we obtain:
\begin{enumerate}
\item the Second-Order \emph{Syntactic} Categorical Algebraic Theory Correspondence, by which second-order algebraic theories and their morphisms correspond to second-order equational presentations and syntactic translations; and
\item the Second-Order \emph{Semantic} Categorical Algebraic Theory Correspondence, by which algebras for second-order equational presentations correspond to second-order functorial models.
\end{enumerate}

\section{First-Order Algebraic Theories}
\label{backg}

\label{lawveretheories}

Lawvere's seminal thesis on algebraic theories \cite{Lawvere2004} develops a presentation--independent category-theoretic formulation of finitary first-order theories; \emph{finitary} in the sense that only operations of arity given by a finite cardinal are considered, and \emph{first-order} in that the arguments of the operations do not allow variable-binding. 
Under his abstraction, an algebraic theory is a functor from a \emph{base category} to a small category with strict finite products, whose morphishms can be thought of as tuples of derived operations. The base category represents the most fundamental equational theory, the theory of equality. It arises from the universal property of the categorical cartesian product. 
We review Lawvere's categorical approach to universal algebra and its syntactic counterpart given by mono-sorted equational presentations.  \\

\textbf{The first-order theory of equality.} Let $\ff$ be the category of finite cardinals and all functions between them. The objects of $\ff$ are denoted by $n \in \N$; it comes equipped with a cocartesian structure given via cardinal sum $m+n$. $\ff$ can be universally characterised as the free cocartesian category generated by the object 1. By duality, the opposite of $\ff$, which we shall denote by $\F$ for $\F$awvere, is equipped with finite products. This category, together with a suitable cartesian functor, form the main constituents of a Lawvere theory.

\begin{defn}[Lawvere theory]
A \emph{Lawvere theory} consists of a small category $\Law$ with strictly associative finite products, together with a strict cartesian identity-on-objects functor $L \colon \F \To \Law$. A morphism of Lawvere theories $L \colon \F \To \Law$ and $L' \colon \F \To \Law '$ is a cartesian functor $F \colon \Law \To \Law '$ which commutes with the theory functors $L$ and $L'$. We write $\LAW$ for the category of Lawvere theories and their morphisms.
\end{defn}

For a Lawvere theory $L \colon \F \To \Law$, the objects of $\Law$ are then precisely those of $\F$. For any $n \in \N$, morphisms in $\Law(n,1)$ are referred to as the \emph{operators} of the theory, and those arising from $\F(n,1)$ as the \emph{elementary} such operators. For any $n,m \in \N$, morphisms in $\Law(n,m)$ are $m$-tuples of operators, because $\Law(n,m) \cong \Law(n,1)^m$. Intuitively, a morphism of Lawvere theories encapsulates the idea of interpreting one theory in another.

\begin{defn}[Functorial models]
A \emph{functorial model} of a Lawvere theory $L \colon \F \To \Law$ in a cartesian category $\cat{C}$ is a cartesian functor $\Law \To \cat{C}$. 
\end{defn}

\textbf{First-order equational presentations}
are the syntactic counterpart of Lawvere theories. 
An equational presentation consists of a signature defining its operations and a set of axioms describing the equations it should obey. Formally, a \emph{mono-sorted first-order equational presentation} is specified as $\E = (\Sigma, E)$, where $\Sigma = \{ \Sigma_n \}_{n \in \N}$ is an indexed family of first-order operators. For a given $n \in \N$, we say that an operator $\omega \in \Sigma_n$ has \emph{arity} $n$. The set of terms $T_\Sigma(V)$ on a set of variables $V$ generated by the signature $\Sigma$ is built up by the grammar
$$ t \in T_\Sigma(V) \squad \Def \squad v \quad | \quad \omega(t_1, \dots, t_k) \quad , $$
where $v \in V$, $\omega \in \Sigma_k$, and for $i =1,\dots,k$, $t_i \in T_\Sigma(V)$. An equation is simply given by a pair of terms, and the set $E$ of the equational presentation $\E = (\Sigma,E)$ contains equations, which we refer to as the axioms of $\E$.

\begin{defn}[First-order syntactic translations]
There are two constituents defining the notion of morphism of first-order equational presentations $\E = (\Sigma, E) \To \E ' = (\Sigma ', E')$. An operator $\omega$ of $\Sigma$ is mapped to a term $\Gamma \vdash t$ of $\Sigma '$, with its context $\Gamma$ given by the arity of $\omega$. This induces a mapping between the terms of $\Sigma$ and $\Sigma '$ in such a way that the axioms of $E$ are respected. Equational presentations are their syntactic presentations form a category, denoted by $\textbf{FOEP}$.
\end{defn}

Indeed, a syntactic morphism with these properties mirrors the behaviour of morphisms of first-order algebraic theories. Notions of mappings of signatures and presentations have been developed in the first-order setting by Fujiwara \cite{Fujiwara1959, Fujiwara1960}, Goguen et al. \cite{GoguenThatcherWagner1978}, and Vidal and Tur \cite{VidalTur2008}, all of which use the common definition that a syntactic morphism maps \emph{operators to terms}. \\

\textbf{Set-theoretic semantics.}
The model-theoretic universe of first-order languages is classically taken to be the category $\Set$. A (set-theoretic) algebra in this universe for a first-order signature $\Sigma$ is a pair $(X,\alg{-}{X})$ consisting of a set $X$ and interpretation functions $\alg{\omega}{X} \colon X^{|\omega|} \To X$, where $|\omega|$ denotes the arity of $\omega$. Algebras induce interpretations on terms (see for example \cite{FioreHur2008a} for details). An algebra for an equational presentation $\E = (\Sigma,E)$ is an algebra for $\Sigma$ which satisfies all equations in $E$, in the sense that an equal pair of terms induces equal interpretation functions in $\Set$.

\subsection{First-Order Categorical Algebraic Theory Correspondence} 

The passage from Lawvere theories and their functorial models to mono-sorted first-order equational presentations and their algebras is invertible, making Lawvere theories an abstract, presentation-independent formalisation of equational presentations. Any first-order equational presentation induces an algebraic theory, and, vice versa, any algebraic theory has an underlying equational presentation. Moreover, morphisms of Lawvere theories correspond to syntactic translations of equational presentations, which gives the following result.

\begin{thm}
The categories $\LAW$ and $\mathbf{FOEP}$ are equivalent.
\end{thm}

The semantic component of the Categorical Algebraic Theory Correspondence given by the equivalence between functorial models for first-order algebraic theories, algebras for first-order equational presentations, and Eilenberg-Moore algebras for finitary monads. We refer the reader to \cite{Borceux1994} for detailed proofs.

\begin{prop}
For every $S$-sorted first-order equational presentation $\E$, there exists a finitary monad $\mon$ on $\Set^S$ such that the category of $\E$-algebras is isomorphic to that of Eilenberg-Moore algebras for $\mon$. Also, for a set $S$ and every finitary monad $\mon$ on $\Set^S$, there exists a first-order algebraic theory $L \colon \F_S \To \Law$ such that the category of Eilenberg-Moore algebras for $\mon$ is isomorphic to the category of functorial models $\FMod(L, \Set^S)$.
\end{prop}

\section{Second-Order Syntax and Semantics}
\label{so_syntax}

The passage from first to second order involves extending the language with both \emph{variable-binding operators} and \emph{parameterised metavariables}. Second-order operators bind a list of variables in each of their arguments, leading to syntax up to alpha equivalence \cite{Aczel1978}. On top of variables, second-order languages come equipped with parameterised metavariables. These are essentially second-order variables for which substitution also involves instantiation. Variable-binding constructs are at the core of fundamental calculi and theories in computer science and logic \cite{Church1936, Church1940}. Examples of second-order languages include the $\lambda$-calculus~\cite{Aczel1978}, the fixpoint operator \cite{Klop1993}, the primitive recursion operator \cite{Aczel1978}, the universal and existential quantifiers of predicate logic \cite{Aczel1980}, and the list iterator~\cite{Raamsdonk2003}.

Over the past two decades, many formal frameworks for languages with binding have been developed, including higher-order abstract syntax \cite{PfenningElliott1988} and Gabbay and Pitts' set-theoretic abstract syntax \cite{GabbayPitts2001}. We review the second-order framework of Fiore et al. \cite{FiorePlotkinTuri1999}, as developed further by Hamana \cite{Hamana2005}, Fiore \cite{Fiore2008}, and Fiore and Hur \cite{FioreHur2010}. 

\subsection{Second-order signatures}

Following the development of Aczel \cite{Aczel1978}, a (mono-sorted) \emph{second-order signature} $\Sigma = (\Omega, |-|)$ is specified by a set of operators $\Omega$ and an arity function $|-| \colon \Omega \To \N^*$. For an operator $\omega \in \Omega$, we write $\omega \colon \tuple{n}{k}$ whenever it has arity $| \omega | = \tuple{n}{k}$. The intended meaning here is that the operator $\omega$ takes $k$ arguments binding $n_i$ variables in the $i^{\mathrm{th}}$ argument. 

Any language with variable binding fits this formalism, including languages with quantifiers \cite{Aczel1980}, a fixpoint operator \cite{Klop1993}, and the primitive recursion operator \cite{Aczel1978}. The most prototypical of all second-order languages is the $\lambda$-calculus.

\begin{ex}
The second-order signature $\Sigma_\lambda$ of the mono-sorted \emph{$\lambda$-calculus} has operators $\tt{abs} \colon (1)$ and $\tt{app} \colon (0,0)$ representing $\lambda$ abstraction and application, respectively.
\end{ex}

\subsection{Second-order terms}

Second-order terms have \emph{metavariables} on top of variables as building blocks. We use the notational convention of denoting variables similar to first-order variables by $x,y,z$, and metavariables by $\mvar{m}, \mvar{n}, \mvar{l}$. Metavariables come with an associated natural number arity, also referred to as its \emph{meta-arity}. A metavariable $\mvar{m}$ of meta-arity $m$, denoted by $\mvar{m} \colon [m]$, is to be parameterised by $m$ terms. 

Second-order terms are considered in contexts with two \emph{zones}, each respectively declaring metavariables and variables. We use the following representation for contexts
$\mcon{m}{k} \Sep \vcon{x}{n}$
where the metavariables $\mvar{m}_i$ and variables $x_j$ are assumed to be distinct. 

Terms are built up by means of operators
from both variables and metavariables, and hence referred to as second-order.
The judgement for \emph{second-order terms} in context
$ \Theta \Sep \Gamma \vdash t $
is defined similar to the second-order syntax of Aczel \cite{Aczel1978} by the following rules.
\begin{itemize}
\setlength\itemindent{50pt}
\item[(Variables)] For $x \in \Gamma$,
$$ \frac{}{\Theta \Sep \Gamma \vdash x} $$
\item[(Metavariables)] For $(\mvar{m} \colon [m]) \in \Theta$,
$$ \frac{\Theta \Sep \Gamma \vdash t_i \quad (1 \leq i \leq m)}{\Theta \Sep \Gamma \vdash \mvar{m}[t_1, \dots, t_m]} $$
\item[(Operators)] For $\omega \colon \tuple{n}{k}$,
$$ \frac{\Theta \Sep \Gamma, \vvar{x}_i \vdash t_i \quad (1 \leq i \leq k)}{\Theta \Sep \Gamma \vdash \omega \big( (\vvar{x}_1)t_1, \dots, (\vvar{x}_k)t_k \big)} $$
where $\vvar{x}_i$ stands for $\ivcon{x}{n}{i}$. 
\end{itemize}

Terms derived according to the first two rules only via variables and metavariables are referred to as \emph{elementary}. Hence, an empty signature with an empty set of operators generates only elementary terms. 

Terms are considered up to the $\alpha$-equivalence relation induced by stipulating
that, for every operator $\omega \colon \tuple{n}{k}$, the variables $\vvar{x}_i$ in the term $\omega \big( (\vvar{x}_1)t_1, \dots, (\vvar{x}_k)t_k \big)$ are bound in $t_i$.

\begin{ex}
Two sample terms for the signature $\Sigma_\lambda$ of the mono-sorted $\lambda$-calculus are $\mvar{m} \colon [1], \mvar{n} \colon [0] \Sep \emptycon \vdash \tt{app} \big(\tt{abs} \big( (x) \mvar{m} [x] \big), \mvar{n}[] \big)$ and
$\mvar{m} \colon [1], \mvar{n} \colon [0] \Sep \emptycon \vdash \mvar{m}[\mvar{n}[]]$.
\end{ex}

\subsection{Second-order substitution calculus}

The second-order nature of the syntax requires a two-level substitution calculus. Each level respectively accounts for the substitution of variables and metavariables, with the latter operation depending on the former \cite{Aczel1978, Klop1993, Raamsdonk2003, Fiore2008}.

\begin{defn}[Substitution]
The operation of capture-avoiding simultaneous \emph{substitution} of terms for variables maps
$ \Theta \Sep \vcon{x}{n} \vdash t$ and $\Theta \Sep \Gamma \vdash t_i \quad (1 \leq i \leq n) $
to 
$ \Theta \Sep \Gamma \vdash t \big\{ x_i \Def t_i \big\}_{i \in \card{n}} $
according to the following inductive definition:
\begin{itemize}
\item[-] $x_j \big\{ x_i \Def t_i \big\}_{i \in \card{n}} = t_j$
\item[-] $\big( \mvar{m} [\dots,s,\dots] \big) \big\{ x_i \Def t_i \big\}_{i \in \card{n}} = \mvar{m} \big[ \dots, s\big\{ x_i \Def t_i \big\}_{i \in \card{n}}, \dots \big] $
\item[-] $\big( \omega ( \dots, (\vcon{y}{k})s ,\dots ) \big) \big\{ x_i \Def t_i \big\}_{i \in \card{n}} \squad = \squad \omega \big( \dots, (\vcon{y}{k}) s \big\{ x_i \Def t_i, y_j \Def z_j \big\}_{i \in \card{n}, j \in \card{k}} , \dots \big)$ 
with $z_j \notin \mathrm{dom}(\Gamma)$ for all $j \in \card{k}$. 
\end{itemize}
\end{defn}

\begin{defn}[Metasubstitution] 
The operation of \emph{metasubstitution} of abstracted terms for metavariables maps
$ \mcon{m}{k} \Sep \Gamma \vdash t $ and $ \Theta \Sep \Gamma , \vvar{x}_i \vdash t_i \quad (1 \leq i \leq k) $
to 
$\Theta \Sep \Gamma \vdash t \big\{ \mvar{m}_i \Def (\vvar{x}_i) t_i \big\}_{i \in \card{k}} $
according to the following inductive definition:
\begin{itemize}
\item[-] $x \big\{ \mvar{m}_i \Def (\vvar{x}_i) t_i \big\}_{i \in \card{k}} = x$
\item[-] $\big( \mvar{m}_l [s_1, \dots, s_{m_l}] \big) \big\{ \mvar{m}_i \Def (\vvar{x}_i) t_i \big\}_{i \in \card{k}} = t_l \Big\{ x^{(i)}_j \Def s_j \big\{ \mvar{m}_i \Def (\vvar{x}_i) t_i \big\}_{i \in \card{k}} \Big\}_{j \in \card{m_l}} $
\item[-] $\big( \omega (\dots, (\vvar{x})s ,\dots) \big) \big\{ \mvar{m}_i \Def (\vvar{x}_i) t_i \big\}_{i \in \card{k}} = \omega \big( \dots, (\vvar{x})s \big\{ \mvar{m}_i \Def (\vvar{x}_i) t_i \big\}_{i \in \card{k}} ,\dots \big)$ 
\end{itemize}
\end{defn}

The operation of metasubstitution is well-behaved, in the sense that it is compatible with substitution (\emph{Substitution-Metasubstitution Lemma}) and monoidal, meaning that it is associative (\emph{Metasubstitution Lemma I}) and has a unit (\emph{Metasubstitution Lemma II}). Formulations of these Lemmas are given in Appendix A, and a detailed proof can be found in \cite{Mahmoud2011}.

\subsection{Parameterisation}

Every second-order term $\Theta \Sep \Gamma \vdash t$ can be \emph{parameterised} to yield a term $ \Theta, \param{\Gamma} \Sep \emptycon \vdash \param{t} $, where for $\Gamma = \vcon{x}{n}$,
$ \param{\Gamma} = \mvar{x}_1 \colon [0], \dots, \mvar{x}_n \colon [0]$ and $ \param{t} = t \big\{ x_i \Def \mvar{x}_i [] \big\}_{i \in \card{n}}$. The variable context is thus replaced under parameterisation by a metavariable context, yielding an essentially equivalent term (formally \emph{parameterised} term) where all its variables are replaced by metavariables, which do not themselves parameterise any terms. This allows us to intuively think of metavariables of zero meta-arity as variables, and vice versa.

\subsection{Second-Order Equational Logic}

A second-order \emph{equation} is given by a pair of second-order terms $\Theta \Sep \Gamma \vdash s$ and $\Theta \Sep \Gamma \vdash t$ in context, written as
$ \Theta \Sep \Gamma \vdash s \eq t $.
A second-order \emph{equational presentation} $\E = (\Sigma, E)$ is specified by a second-order signature $\Sigma$ together with a set of equations $E$, the \emph{axioms} of the presentation $\E$, over it. Axioms are usually denoted by
$ \Theta \Sep \Gamma \vdash_E t \eq s $
to distinguish them from any other equations.

\begin{ex}
The equational presentation $\E_\lambda = (\Sigma_\lambda, E_\lambda)$ of the mono-sorted $\lambda$-calculus extends the second-order signature $\Sigma_\lambda$ with the following axioms.
\begin{itemize}
\item[] \hspace{10pt} $(\beta) \quad \mvar{m} \colon [1], \mvar{n} \colon [0] \Sep \emptycon \vdash_{E_\lambda} \tt{app} \big( \tt{abs} \big( (x) \mvar{m} [x] \big), \mvar{n}[] \big) \squad \eq \squad \mvar{m} \big[ \mvar{n}[] \big] $ 
\item[] \hspace{10pt} $(\eta) \quad  \mvar{f} \colon [0] \Sep \emptycon \vdash_{E_\lambda} \tt{abs} \big( (x) \tt{app} (\mvar{f}[],x) \big) \squad \eq \squad \mvar{f}[] $ 
\end{itemize}
\end{ex}

It is worth emphasising that the (mono-sorted) $\lambda$-calculus is merely taken as a running example throughout this work, for it is the most intuitive and widely-known such calculus. The expressiveness of the second-order formalism does not, however, rely exclusively on that of the $\lambda$-calculus. One can directly axiomatise, say,  primitive recursion \cite{Aczel1978} and predicate logic \cite{Plotkin1998} as second-order equational presentations.

The rules of \emph{Second-Order Equational Logic} are given in Figure 1. Besides the rules for axioms and equivalence, the logic consists of just one additional rule stating that the operation of metasubstitution in extended metavariable context is a congruence. 
The expressive power of this system can be seen through the following two sample derivable rules.\\

(Substitution)
$$ \frac{\Theta \Sep x_1, \dots, x_n \vdash s \eq t \qquad \Theta \Sep \Gamma \vdash s_i \eq t_i \quad (1 \leq i \leq n)}{\Theta \Sep \Gamma \vdash s \{ x_i \Def s_i \}_{i \in \card{n}} \eq t \{ x_i \Def t_i \}_{i \in \card{n}}} $$

(Extension)
$$ \frac{\mcon{m}{k} \Sep \Gamma \vdash s \eq t}{\mvar{m}_1 \colon [m_1+n], \dots, \mvar{m}_k \colon [m_k+n] \Sep \Gamma, \vcon{x}{n} \vdash s^\# \eq t^\#}  $$

\vspace{7pt}
where $u^\# = u \{ \mvar{m}_i \Def \tuple{x}{n} \mvar{m}_i[y^{(i)}_1, \dots, y^{(i)}_{m_i},\vcon{x}{n}] \}_{i \in \card{k}}$.\\

\begin{figure}[t]
(Axioms)
$$ \frac{\Theta \Sep \Gamma \vdash_E s \eq t}{\Theta \Sep \Gamma \vdash s \eq t} $$

(Equivalence)
$$ \frac{\Theta \Sep \Gamma \vdash t}{\Theta \Sep \Gamma \vdash t \eq t} \qquad \frac{\Theta \Sep \Gamma \vdash s \eq t}{\Theta \Sep \Gamma \vdash t \eq s} \qquad \frac{\Theta \Sep \Gamma \vdash s \eq t \qquad \Theta \Sep \Gamma \vdash t \eq u}{\Theta \Sep \Gamma \vdash s \eq u} $$

(Extended metasubstitution)
$$ \frac{\mcon{m}{k} \Sep \Gamma \vdash s \eq t \qquad \Theta \Sep \Delta, \vvar{x}_i \vdash s_i \eq t_i \quad (1 \leq i \leq k)}{\Theta \Sep \Gamma, \Delta \vdash s \big\{ \mvar{m}_i \Def (\vvar{x}_i)s_i \big\}_{i \in \card{k}} \eq t \big\{ \mvar{m}_i \Def (\vvar{x}_i)t_i \big\}_{i \in \card{k}}} $$
\caption{Second-Order Equational Logic}
\vspace{5pt}
\rule[1ex]{\textwidth}{0.5pt}
\end{figure}

Performing the operation of parameterisation on a set of equations $E$ to obtain a set of \emph{parameterised equations} $\param{E}$, we have that all of the following are equivalent:
$$ \Theta \Sep \Gamma \vdash_E s \eq t \quad , \qquad \Theta , \param{\Gamma} \Sep \emptycon \vdash_\E \param{s} \eq \param{t} $$
$$ \Theta \Sep \Gamma \vdash_{\param{E}} s \eq t \quad, \qquad \Theta , \param{\Gamma} \Sep \emptycon \vdash_{\param{E}} \param{s} \eq \param{t} $$
Hence, without loss of generality, any set of axioms can be transformed into a parameterised set of axioms, which in essence represents the same equational presentation. One may restrict to axioms containing empty variable contexts as in the \emph{CRS}s of Klop \cite{Klop1980}, but there is no reason for us to do the same.

\subsection{Second-Order Universal Algebra}

The model theory of Fiore and Hur \cite{FioreHur2010} for second-order equational presentations is recalled. For our purposes, this is presented here in elementary concrete model-theoretic terms rather than in abstract monadic terms. The reader is referred to \cite{FioreHur2010} for the latter perspective. \\

\textbf{Semantic universe.}
Recall that we write $\ff$ for the free cocartesian category on an object. Explicitly, $\ff$ has $\N$ as set of objects and morphisms $m \To n$ given by functions $\card{m} \To \card{n}$. 
The second-order model-theoretic development lies within the semantic universe $\Set^\ff$, the presheaf category of sets in variable contexts \cite{FiorePlotkinTuri1999}. It is a well-known category, and the formalisation of second-order model theory relies on some of its intrinsic properties. In particular, $\Set^\ff$ is bicomplete with limits and colimits computed pointwise \cite{MaclaneMoerdijk1992}. 
We write $\y$ for the Yoneda embedding $ \ff^\op \hookrightarrow \Set^\ff$. \\

\textbf{Substitution.} We recall the \emph{substitution monoidal structure} in the semantic universe $\Set^\ff$ as presented in \cite{FiorePlotkinTuri1999}. The unit is given by the \emph{presheaf of variables} $\y 1$, explicitly the embedding $\ff \hookrightarrow \Set$. This object is a crucial element of the semantic universe $\Set^\ff$, as it provides an arity for variable binding. The monoidal tensor product $X \tensor Y$ of presheaves $X,Y \in \Set^\ff$ is given by
$$ X \tensor Y = \int^{k \in \ff} X(k) \times Y^k \quad . $$

A monoid \vspace{-0.6cm}
\begin{diagram}[small]
\y 1 & \rTo^{\nu} & A & \lTo^{\varsigma} & A \tensor A
\end{diagram}
for the substitution monoidal structure equips $A \in \Set^\ff$ with substitution structure. In particular, the map $\nu_k \colon \y k \To A^k$, defined as the composite \vspace{-0.6cm}
\begin{diagram}[small]
\y k \cong (\y 1)^k & \rTo^{\nu^k} & A^k \quad , 
\end{diagram}
induces the embedding
$ (A^{\y n} \times A^n)(k) \To A(k+n) \times A^k(k) \times A^n(k) \To (A \tensor A)(k) $,
which, together with the multiplication, yield a \emph{substitution operation}
$ \varsigma_n \colon A^{\y n} \times A^n \To A $
for every $n \in \N$. These substitution operations provide the interpretations of metavariables. \\

\textbf{Algebras.} Every second-order signature $\Sigma = (\Omega, |-|)$ induces a \emph{signature endofunctor} $\sfun{\Sigma} \colon \Set^\ff \To \Set^\ff$ given by
$$ \sfun{\Sigma}X \squad = \coprod_{\omega \colon \tuple{n}{k} \in \Omega} \quad \prod_{i \in \card{k}} X^{\y n_i} \quad . $$
$\sfun{\Sigma}$-algebras $\sfun{\Sigma}X \To X$ provide an interpretation
$$ \alg{\omega}{X} \colon \prod_{i \in \card{k}} X^{\y n_i} \To X $$
for every operator $\omega \colon \tuple{n}{k}$ in $\Sigma$. 
Note that there are canonical natural isomorphisms
\begin{eqnarray*}
\coprod_{i \in I} (X_i \tensor Y) & \cong & \big( \coprod_{i \in I} X_i \big) \tensor Y \\
 \prod_{i \in \card{n}} (X_i \tensor Y) & \cong & \big( \prod_{i \in \card{n}} X_i  \big) \tensor Y
\end{eqnarray*}
and, for all points $\eta \colon \y 1 \To Y$, natural extension maps
$$ \eta^{\# n} \colon X^{\y n} \tensor Y \To (X \tensor Y)^{\y n} \quad . $$
These constructions equip every signature endofunctor $\sfun{\Sigma}$ with a \emph{pointed strength}
$ \varpi_{X,\y 1 \To Y} \colon \sfun{\Sigma}(X) \tensor Y \To \sfun{\Sigma} (X \tensor Y)$.
This property plays a critical role in the notion of algebra with substitution structure, which depends on this pointed strength. The extra structure on a presheaf $Y$ in the form of a point $\varpi \colon \y 1 \To Y$ reflects the need of fresh variables in the definition of substitution for binding operators. We refer the reader to \cite{FiorePlotkinTuri1999} and \cite{Fiore2008} for a detailed development. \\

\textbf{Models.} A model for a second-order signature $\Sigma$ is an algebra equipped with a compatible substitution structure. Formally, $\Sigma$-\emph{models} are defined to be $\Sigma$-\emph{monoids}, which are objects $A \in \Set^\ff$ equipped with an $\sfun{\Sigma}$-algebra structure $\alpha \colon \sfun{\Sigma}A \To A$ and a monoid structure $\nu \colon \y 1 \To A$ and $\varsigma \colon A \tensor A \To A$ that are compatible in the sense that the following diagram commutes.
\begin{diagram}[small]
\sfun{\Sigma}(A) \tensor A & \rTo^{\varpi_{A,\nu}} & \sfun{\Sigma} (A \tensor A) & \rTo^{\sfun{\Sigma} \varsigma} & \sfun{\Sigma}(A) \\
\dTo^{\alpha \tensor A} & & & & \dTo_{\alpha} \\
A \tensor A & & \rTo^{\varsigma} & &A
\end{diagram}

We denote by $\Mod(\Sigma)$ the category of $\Sigma$-models, with morphisms given by maps that are both $\sfun{\Sigma}$-algebra and monoid homomorphisms. \\

\textbf{Soundness and completeness.}
We review the soundness and completeness of the model theory of \emph{Second-Order Equational Logic} as presented in \cite{FioreHur2010}. A model $A \in \Mod(\Sigma)$ for a second-order signature $\Sigma$ is explicitly given by, for a metavariable context $\Theta = (\mcon{m}{k})$ and variable context $\Gamma = (\vcon{x}{n})$, a presheaf
$\alg{\Theta \Sep \Gamma}{A} = \prod_{i \in \card{k}} A^{\y m_i} \times \y n $
of $\Set^\ff$, together with interpretation functions 
$\alg{\omega}{A} \colon \prod_{j \in \card{l}} A^{\y n_j} \To A $
for each operator $\omega \colon \tuple{n}{l}$ of $\Sigma$. This induces the interpretation of a second-order term $\Theta \Sep \Gamma \vdash t $ in $A$ as a morphism
$ \alg{\Theta \Sep \Gamma \vdash t}{A} \colon \alg{\Theta \Sep \Gamma}{A} \To A $
in $\Set^\ff$, which is given by structural induction as follows: 
\begin{itemize}
\item[-] $\alg{\Theta \Sep \Gamma \vdash x_i}{A}$ is the composite
\begin{diagram}[small]
\alg{\Theta \Sep \Gamma}{A} & \rTo^{\pi_2} & \y n & \rTo^{\nu_n} &A^n & \rTo^{\pi_j} & A \quad.
\end{diagram}
\item[-] $\alg{\Theta \Sep \Gamma \vdash \mvar{m}_i[t_1, \dots, t_{m_i}]}{A}$ is the composite
\begin{diagram}[small]
\alg{\Theta \Sep \Gamma}{A} & \rTo^{\langle \pi_i \pi_1 , f \rangle} & A^{\y m_i} \times A^{m_i} & \rTo^{\varsigma_{m_i}} & A \quad ,
\end{diagram}
where $f = \big\langle \alg{\Theta \Sep \Gamma \vdash t_j}{A} \big\rangle_{j \in \card{m_i}}$.
\item[-] For an operator $\omega \colon \tuple{n}{l}$ of $\Sigma$,
$ \alg{\Theta \Sep \Gamma \vdash \omega \big( (\vvar{y}_1) t_1, \dots, (\vvar{y}_l) t_l \big)}{A} $
is the composite
\begin{diagram}[small]
\alg{\Theta \Sep \Gamma}{A} & \rTo^{\langle f_j \rangle_{j \in \card{l}}} & \prod_{j \in \card{l}} A^{\y n_j} & \rTo^{\alg{\omega}{A}} & A \quad,
\end{diagram}
where $f_j$ is the exponential transpose of
\begin{diagram}[small]
\prod_{i \in \card{k}} A^{\y m_i} \times \y n \times \y n_j \cong \prod_{i \in \card{k}} A^{\y m_i} \times \y (n+n_j) & \rTo^{\alg{\Theta \Sep \Gamma, \vvar{y}_j \vdash t_j}{A}} & A \quad . \\
\end{diagram}
\end{itemize}

 A model $A \in \Mod(\Sigma)$ \emph{satisfies} an equation $\Theta \Sep \Gamma \vdash s \eq t$, which we write as $A \models (\Theta \Sep \Gamma \vdash s \eq t)$, if and only if $\alg{\Theta \Sep \Gamma \vdash s}{A} = \alg{\Theta \Sep \Gamma \vdash t}{A}$ in $\Set^\ff$. 
For a second-order equational presentation $\E = (\Sigma,E)$, the category $\Mod(\E)$ of $\E$-models is the full subcategory of $\Mod(\Sigma)$ consisting of the $\Sigma$-models that satisfy the axioms $E$. 

\begin{thm}[Second-Order Soundness and Completeness]
For a second-order equational presentation $\E = (\Sigma,E)$, the judgement $\Theta \Sep \Gamma \vdash s \eq t$ is derivable from $E$ if and only if $A \models (\Theta \Sep \Gamma \vdash s \eq t)$ for all $\E$-models $A$.
\end{thm}

At the level of equational derivability, the extension of (first-order) universal algebra to the second-order framework, as presented in this chapter, is conservative. 
Clearly, every first-order signature is a second-order signature in which all operators do not bind any variables in their arguments. Any first-order term $\Gamma \vdash t$ can therefore be represented as the second-order term $\emptycon \Sep \Gamma \vdash t$. Indeed, for a set of first-order equations, if the equation $\Gamma \vdash s \eq t$ is derivable in first-order equational logic, then its corresponding second-order representative $\emptycon \Sep \Gamma \vdash s \eq t$ is derivable in second-order equational logic. 
The converse statement is what is known as \emph{conservativity} of second-order equational derivability. Although this result is not directly utilised in our work, we recall it for the benefit of comprehensiveness, and refer the reader to \cite{FioreHur2010} for the proof.

\begin{thm}[Conservativity]
Second-Order Equational Logic (Figure 1) is a conservative extension of First-Order Equational Logic. More precisely, if a second-order equation between first-order terms \linebreak $\emptycon \Sep \Gamma \vdash s \eq t$ lying in an empty metavariable context is derivable in second-order equational logic, then $\Gamma \vdash s \eq t$ is derivable in first-order equational logic.
\end{thm}

\section{The Second-Order Theory of Equality}
\label{so_equality}

In categorical algebraic theory, the elementary theory of equality plays a pivotal role, as it represents the most fundamental such theory. We identify the second-order algebraic theory of equality $\M$. This we do first in syntactic terms, via an explicit description of its categorical structure, and in abstract terms by establishing its universal property. 
Just as the cartesian structure characterises first-order algebraic theories, we will show that \emph{exponentiability} abstractly formalises essential second-order characteristics. 

\subsection{Categorical exponentials}

For $\cat{C}$ a cartesian category and $A,B$ objects of $\cat{C}$, an \emph{exponential object} $\e{A}{B}$ is a universal morphism from $- \times A \colon \cat{C} \To \cat{C}$ to $B$. Explicitly, $\e{A}{B}$ comes equipped with a morphism $\ev \colon (\e{A}{B}) \times A \To B$ such that for any object $C$ of $\cat{C}$ and $f \colon C \times A \To B$, there is a unique $\lam(f) \colon C \To \e{A}{B}$, the \emph{exponential mate} of $f$, making $\ev \circ (\lam(f) \times A) = f$. 
A cartesian functor $F \colon \cat{C} \To \cat{D}$ is \emph{exponential} if it preserves the exponential structure in $\cat{C}$. Formally, for any exponential $\e{A}{B}$ in $\cat{C}$, $\e{FA}{FB}$ is an exponential object in $\cat{D}$ and the exponential mate of \vspace{-0.1cm}
\begin{diagram}[small]
F(\e{A}{B}) \times FA \cong F((\e{A}{B}) \times A) & \rTo^{F\ev} & FB
\end{diagram} 
is an isomorphism $F(\e{A}{B}) \To \e{FA}{FB}$. 
An object $C$ in a cartesian category $\cat{C}$ is \emph{exponentiable} if for all objects $D \in \cat{C}$ the exponential $\e{C}{D}$ exists in $\cat{C}$. Given an exponentiable object $C$, the $n$-ary cartesian product $C^n$ is obviously exponentiable for all $n \in \N$.

\subsection{The Second-Order Theory of Equality}

The syntactic viewpoint of second-order theories leads us to define the category $\M$ with set of objects given by $\N^*$ and morphisms $\tuple{m}{k} \To \tuple{n}{l}$ given by tuples
$$ \big\langle \mcon{m}{k} \Sep \vcon{x}{n_i} \vdash t_i \big\rangle_{i \in \card{l}} $$
of elementary terms under the empty second-order signature. 
The identity on $\tuple{m}{k}$ is given by
$$ \big\langle \mcon{m}{k} \Sep \vcon{x}{m_i} \vdash \mvar{m}_i[\vcon{x}{m_i}] \big\rangle_{i \in \card{k}} \quad ;$$
whilst the composition of 
$$ \big\langle \mcon{l}{i} \Sep \vcon{x}{m_p} \vdash s_p \big\rangle_{p \in \card{j}} \colon \tuple{l}{i} \To \tuple{m}{j} $$
and 
$$ \big\langle \mcon{m}{j} \Sep \vcon{y}{n_q} \vdash t_q \big\rangle_{q \in \card{k}} \colon \tuple{m}{j} \To \tuple{n}{k} $$
is given via metasubstitution by
$$ \big\langle \mcon{l}{i} \Sep \vcon{y}{n_q} \vdash t_q \{ \mvar{m}_p \Def (\vcon{x}{m_p}) s_p \}_{p \in \card{j}} \big\rangle_{q \in \card{k}} \colon \tuple{l}{i} \To \tuple{n}{k} \quad . $$

The category $\M$ is well-defined, as the identity and associativity axioms hold because of intrinsic properties given by the Metasubstitution Lemmas.
It comes equipped with a strict cartesian structure, with the terminal object given by the empty sequence $()$, the terminal map $\tuple{m}{k} \To ()$ being the empty tuple $\langle \rangle$, and the binary product of $\tuple{m}{k}$ and $\tuple{n}{l}$ given by their concatenation $(m_1, \dots, m_k, n_1, \dots, n_l)$. Any object $\tuple{m}{k}$ is thus the cartesian product of the single tuples $(m_i)$, for $i \in \card{k}$, with projections
$$ \big\langle \mcon{m}{k} \Sep \ivcon{x}{m}{i} \vdash \mvar{m}_i [\ivcon{x}{m}{i}] \big\rangle \colon \tuple{m}{k} \To (m_i) \quad . $$
Indeed, given morphisms
$$ \big\langle \mcon{n}{l} \Sep \ivcon{x}{m}{i} \vdash q_i \big\rangle \colon \tuple{n}{l} \To (m_i) $$
for $i \in \card{k}$, the mediating morphism is
$$ \big\langle \mcon{n}{l} \Sep \ivcon{x}{m}{i} \vdash q_i \big\rangle_{i \in \card{k}} \quad . $$

\subsection{Exponential structure} 

In $\M$, the object $(0)$ is exponentiable. For any tuple $\tuple{m}{k}$, the exponential $\e{(0)}{\tuple{m}{k}}$ is given by $\etuple{m}{k}{1}$, with evaluation map $\ev_{\overrightarrow{m},1} \colon \etuple{m}{k}{1} \times (0) \To \tuple{m}{k}$ given by the $k$-tuple
$$ \big\langle\emcon{m}{k}{1}, \mvar{n}:[0] \Sep \ivcon{x}{m}{i} \vdash \mvar{m}_i \big[ \ivcon{x}{m}{i}, \mvar{n}[] \big] \big\rangle_{i \in \card{k}} \quad . $$

For any $\tuple{n}{l}$, the exponential mate $\lam (\langle t_i \rangle_{i \in \card{k}})$ of a map 
$$ \big\langle \mcon{n}{l}, \mvar{m}:[0] \Sep \ivcon{x}{m}{i} \vdash t_i \big\rangle_{i \in \card{k}} \colon \tuple{n}{l} \times (0) \To \tuple{m}{k} $$
is given by
$ \big\langle \mcon{n}{l} \Sep \ivcon{x}{m}{i}, y_i \vdash t_i \{ \mvar{m} \Def y_i \} \big\rangle_{i \in \card{k}}$.
More generally, for any $n \in \N$, the exponential $\e{(0)^n}{\tuple{m}{k}}$ is given by the tuple $ \etuple{m}{k}{n}$. 

The exponential structure in $\M$ embodies attributes intrinsic to second-order languages. First, note that for each $n \in \N$, the \emph{metaweakening} operation $W_n \colon \M \To \M$ mapping $\tuple{m}{k}$ to $\etuple{m}{k}{n}$, and a morphism $\tuple{m}{k} \To \tuple{n}{l} $ of the form
$$ \big\langle \mcon{m}{k} \Sep \ivcon{y}{n}{j} \vdash t_j \big\rangle_{j \in \card{l}} $$
to
\begin{eqnarray*}
&\quad& \Big\langle  \emcon{m'}{k}{n} \Sep \ivcon{y}{n}{j}, z_1^{(j)}, \dots, z_n^{(j)} \vdash \\ &\qquad& \hspace{90pt}  t_j \left\{ \mvar{m}_i \Def (\ivcon{x}{m}{i}) \mvar{m'}_i \big[ \ivcon{x}{m}{i}, z_1^{(j)}, \dots, z_n^{(j)} \big] \right\}_{i \in \card{k}} \Big\rangle_{j \in \card{l}}
\end{eqnarray*}
is in fact the right adjoint $\e{(0)^n}{(-)} \colon \M \To \M$ to the functor $(-) \times (0)^n \colon \M \To \M$. 
Moreover, for any $\tuple{m}{k}$, the resulting bijection
$$ \M \big( \tuple{m}{k}, \e{(0)^n}{(0)} \big) \squad \cong \squad \M \big( \tuple{m}{k} \times (0)^n, (0) \big) $$
formalises the correspondence between a second-order term and its parameterisation. 
Abstractly, every morphism $\langle s \rangle \colon \tuple{m}{k} \To (n) $ can be \emph{parameterised} as $\ev_n \circ \big( \langle s \rangle \times (0)^n \big)$, whose exponential mate $\lam \big( \ev_n \circ ( \langle s \rangle \times (0)^n )\big)$ is just $\langle s \rangle$. 
Finally, the exponential structure manifests itself in all second-order terms, which, when viewed as morphisms of $\M$, decompose via unique universal maps.

\begin{lem}
In the category $\M$, every morphism of the form
$$ \big\langle \mcon{m}{k} \Sep \vcon{x}{n} \vdash x_i \big\rangle \colon \tuple{m}{k} \To (n) $$
decomposes as \vspace{-1cm}
\begin{diagram}
\tuple{m}{k} & \rTo^{} & () & \rTo^{\lam(\pi^{(n)}_i \circ \cong)} & (n) \quad ,
\end{diagram}
where the unlabelled morphism is the unique terminal map, and $\lam(\pi^{(n)}_i \circ \cong)$ is the exponential mate of the $i$-th projection $() \times (0)^n \cong (0)^n \stackrel{\pi^{(n)}_i}{\longrightarrow} (0)$. Moreover, every morphism 
$$ \big\langle \mcon{m}{k} \Sep \vcon{x}{n} \vdash \mvar{m}_i [ t_1, \dots, t_{m_i} ] \big\rangle \colon \tuple{m}{k} \To (n) $$
decomposes as \vspace{-1cm}
\begin{diagram}
\tuple{m}{k} & \rTo^{\left\langle \pi_i, t_1, \dots, t_{m_i} \right\rangle} & ( m_i, n^{m_i}) & \rTo^{\varsigma_{m_i, n}} & (n) \quad ,
\end{diagram}
where $n^{m_i}$ denotes the sequence $n, \dots, n$ of length $m_i$, $\varsigma_{m_i,n}$ is the exponential mate of 
\begin{diagram}
(m_i, n^{m_i}) \times (0)^n & \rTo^{(m_i) \times \ev_{m_i, n}} & (m_i) \times (0)^{m_i} & \rTo^{\ev_{m_i}} & (0) \quad ,
\end{diagram}
and $\ev_{m_i,n}$ is the evaluation map associated with the exponential $(\e{(0)^n}{(0)^{m_i}}) = (n)^{m_i}$.
\end{lem}

\begin{proof}
Explicitly, $(\pi_i^{(n)} \circ \cong) \colon () \times (0)^n \To (0)$ is given by
$ \big\langle \mvar{n}_1 \colon [0], \dots, \mvar{n}_n \colon [0] \Sep \emptycon \vdash \mvar{n}_i [] \big\rangle $
and its unique exponential mate is
$ \big\langle \emptycon \Sep \vcon{x}{n} \vdash \mvar{n}_i[] \big\{ \mvar{n}_j \Def x_j \big\}_{j \in \card{n}} \big\rangle $,
which is simply $\big\langle \emptycon \Sep \vcon{x}{n} \vdash x_i \big\rangle$. Composing this with $\langle  \rangle \colon \tuple{m}{k} \To ()$ yields $\big\langle \mcon{m}{k} \Sep \vcon{x}{n} \vdash x_i \big\rangle$. Next, the morphism $\varsigma_{m_i,n} \colon (m_i,n^{m_i}) \To (n)$ is syntactically given by
$$ \big\langle \mvar{m}_i \colon [m_i], \mvar{n}_1 \colon[0], \dots, \mvar{n}_{m_i} \colon[0] \Sep \vcon{x}{n} \vdash \mvar{m}_i \big[ \mvar{n}_1 [\vcon{x}{n}], \dots, \mvar{n}_{m_i} [\vcon{x}{n}] \big] \big\rangle \quad, $$
and thus composed with $\langle \pi_i, t_1, \dots, t_{m_i} \rangle$
\begin{eqnarray*}
\Big\langle \mvar{m}_i \colon [m_i], \mvar{n}_1 \colon[0], \dots, \mvar{n}_{m_i} \colon[0] \Sep \vcon{x}{n} & \vdash & \mvar{m}_i \big[ \mvar{n}_1 [\vcon{x}{n}], \dots, \mvar{n}_{m_i} [\vcon{x}{n}] \big] \\
& \quad & \big\{ \mvar{m}_i \Def (\vcon{y}{m_i}) \mvar{m}_i [\vcon{y}{m_i}] \big\} \\
& \quad & \big\{ \mvar{n}_j \Def (\vcon{x}{n}) t_j \big\}_{j \in \card{m_i}} \Big\rangle \quad ,
\end{eqnarray*}

\vspace{-10pt} 
this equals
$$ \big\langle \mvar{m}_i \colon [m_i], \mvar{n}_1 \colon[0], \dots, \mvar{n}_{m_i} \colon[0] \Sep \vcon{x}{n} \vdash \mvar{m}_i [t_1, \dots, t_{m_i}] \big\rangle \quad . $$
\end{proof}

\subsection{Universal property}

The exponential structure in $\M$ provides a universal semantic characterisation of $\M$. Loosely speaking, $\M$ is the free strict cartesian category on an exponentiable object. We point out the analogy to the first-order theory of equality $\F$, which is the cartesian category freely generated by a single object. 

\begin{prop}[Universal property of $\M$]
The category $\M$, together with the exponentiable object $(0) \in \M$, is initial amongst cartesian categories equipped with an exponentiable object and with respect to cartesian functors that preserve the exponentiable object.
\end{prop}

\begin{proof}
Let $\cat{D}$ be a cartesian category equipped with an exponentiable object $D$. There is a functor $I \colon \M \To \cat{D}$ mapping the tuple $\tuple{m}{k}$ to $(\e{D^{m_1}}{D}) \times \dots \times (\e{D^{m_k}}{D})$, and defined on morphisms of $\M$ by structural induction as follows: 
\begin{itemize}
\item $\qquad \big\langle \mcon{m}{k} \Sep \vcon{x}{n} \vdash x_i \big\rangle \colon \tuple{m}{k} \To (n) \qquad \stackrel{I}{\longmapsto}$ \vspace{-0.3cm}
\begin{diagram}[small]
(\e{D^{m_1}}{D}) \times \dots \times (\e{D^{m_k}}{D}) & \rTo^{!^\cat{D}} & 1 & \rTo^{\lam(\pi_i^\cat{D} \circ \cong)} & (\e{D^n}{D})
\end{diagram}
\item $\qquad \big\langle \mcon{m}{k} \Sep \vcon{x}{n} \vdash \mvar{m}_i [t_1, \dots, t_{m_i}] \big\rangle \colon \tuple{m}{k} \To (n) \qquad \stackrel{I}{\longmapsto}$ \vspace{-0.2cm}
\begin{diagram}[small]
(\e{D^{m_1}}{D}) \times \dots \times (\e{D^{m_k}}{D}) & \rTo^{\big\langle \pi^{\cat{D}}_i, I\langle t_1 \rangle, \dots, I \langle t_{m_i} \rangle \big\rangle} & (\e{D^{m_i}}{D}) \times (\e{D^n}{D})^{m_i} & \rTo^{\varsigma^\cat{D}_{m_i,n}} & (\e{D^n}{D})
\end{diagram}
\end{itemize}

We superscript cartesian and exponential maps by $\cat{D}$ to distinguish them from those in $\M$.
Note that $I$ is cartesian by definition and moreover exponential. To see this, note that 
$$ I \big( \e{(0)}{(m)} \big) \squad = \squad I (m+1) \squad = \squad \e{D^{m+1}}{D} \squad \cong \squad \e{D}{(\e{D^m}{D})} \squad = \squad \e{I(0)}{I(m)} \quad , $$
and that the exponential mate of $I(\ev_{1,m}) \colon (\e{D^{m+1}}{D}) \times D \To (\e{D^m}{D})$ in $\cat{D}$ is the isomorphism 
$$ (\e{D^{m+1}}{D}) \cong \e{D}{(\e{D^m}{D})} \quad . $$

To see that $I$ is indeed the unique (up to isomorphism) universal functor associated with the initiality of $\M$, suppose that we are given a functor $F \colon \M \To \cat{D}$ which is cartesian and exponential mapping $(0)$ to $D$. Then $F$ is isomorphic to $I$. This is evident on objects, as we have
\begin{eqnarray*}
F\tuple{m}{k} &=& F \big( (m_1) \times \dots \times (m_k) \big) \\
&\cong& F(m_1) \times \dots \times F(m_k) \\
&=& F\big(\e{(0)^{m_1}}{(0)}\big) \times \dots \times F\big(\e{(0)^{m_k}}{(0)}\big) \\
&\cong& \big( \e{F(0)^{m_1}}{F(0)} \big) \times \dots \times \big( \e{F(0)^{m_k}}{F(0)} \big) \\
&=& (\e{D^{m_1}}{D}) \times \dots \times (\e{D^{m_k}}{D}) \\
&=& I\tuple{m}{k} \quad .
\end{eqnarray*}
Given a morphism $\langle t \rangle \colon \tuple{m}{k} \To (n)$ of $\M$, the fact that $I\langle t \rangle = F \langle t \rangle$ is an immediate consequence of the cartesian and exponential property of $F$ and $I$. More precisely, by induction on the structure of the term $t$, we have:
\begin{itemize}
\item The map 
$$\big\langle \mcon{m}{k} \Sep \vcon{x}{n} \vdash x_i \big\rangle \colon \tuple{m}{k} \To (n)$$
decomposes as $\lam(\pi_i^\M \circ \cong) \circ !^\M$, and since $F$ preserves the cartesian and exponential structure, $F \big( \lam(\pi_i^\M \circ \cong) \circ !^\M \big) = \lam(\pi_i^\cat{D} \circ \cong) \circ !^\cat{D}$, which is exactly the image under $I$. 
\item Similarly, $\big\langle \mcon{m}{k} \Sep \vcon{x}{n} \vdash \mvar{m}_i [t_1, \dots, t_{m_i}] \big\rangle \colon \tuple{m}{k} \To (n)$ decomposes via universal cartesian and exponential morphisms of $\M$, which are preserved by both $I$ and $F$, and thus their image under them must be equal.
\end{itemize}
\end{proof}

\vspace{-20pt}

\section{Second-Order Algebraic Theories}

We extend Lawvere's fundamental notion of algebraic theory \cite{Lawvere2004} to the second-order universe. Second-order algebraic theories are defined as second-order-structure preserving functors from the category $\M$ to cartesian categories.

\begin{defn}[Second-order algebraic theories]
A \emph{second-order algebraic theory} consists of a small cartesian category $\Mlaw$ and a strict cartesian identity-on-objects functor $M \colon \M \To \Mlaw$ that preserves the exponentiable  object $(0)$. For second-order algebraic theories $M \colon \M \To \Mlaw$ and $M' \colon \M \To \Mlaw '$, a \emph{second-order algebraic translation} is a cartesian functor $F \colon \Mlaw \To \Mlaw '$ such that 
\begin{diagram}[small]
& & \M & & \\
& \ldTo^{M} & & \rdTo^{M'} \\
\Mlaw & & \rTo^{F} & & \Mlaw ' \quad . 
\end{diagram}
We denote by $\SOAT$ the category of second-order algebraic theories and second-order algebraic translations, with the evident identity and composition.
\end{defn}

The most basic example of a second-order algebraic theory is the \emph{second-order algebraic theory of equality} given by the category $\M$ together with the identity functor. We later show that this is in fact the (second-order) algebraic theory corresponding to a second-order presentation with no operators. This is analogous to the theory of sets corresponding to $\F$ in the first-order setting. 

Every second-order algebraic theory has an \emph{underlying first-order algebraic theory}. To formalise this, recall that the first-order algebraic theory of equality $\F$ is the free strict cartesian category on an object and consider the unique cartesian functor $\F \To \M$ mapping the generating object to the generating exponentiable object $(0)$. Then, the first-order algebraic theory underlying a given second-order algebraic theory $\M \To \Mlaw$ is given by $\F \To \Law_\Mlaw$, where $\F \to \Law_\Mlaw \hookrightarrow \Mlaw$ is the identity-on-objects, full-and-faithful factorisation of $\F \To \M \To \Mlaw$. In particular, the first-order algebraic theory of equality $Id_\F \colon \F \To \F$ underlies the second-order algebraic theory of equality $Id_\M \colon \M \To \M$.

\subsection{Second-Order Theory/Presentation Correspondence}

We illustrate how to construct second-order algebraic theories from second-order equational presentations, and vice versa, and prove that these constructions are mutually inverse.  \\

\textbf{The theory of a presentation.} For a second-order equational presentation $\E = (\Sigma, E)$, the \emph{classifying category} $\M(\E)$ has a set of objects $\N^*$ and morphisms $\tuple{m}{k} \To \tuple{n}{l}$ given by tuples
$$ \big\langle \big[ \mcon{m}{k} \sep \ivcon{x}{n}{i} \vdash t_i \big]_\E \big\rangle_{i \in \card{l}} $$
of equivalence classes of terms generated from $\Sigma$ under the equivalence relation identifying two terms if and only if they are provably equal in $\E$ from \emph{Second-Order Equational Logic} (Figure 1). Identities and composition are defined on representatives as in $\M$. Indeed, composition via metasubstitution respects the equivalence relation, as for 
$$\mcon{m}{k} \sep \vcon{x}{n} \vdash_\E t_1 \eq t_2 \qquad \textrm{and} \qquad \mvar{n} \colon [n] \vdash \vcon{y}{l} \vdash_\E s_1 \eq s_2$$
the equality
$$ \mcon{m}{k} \sep \vcon{y}{l} \vdash_\E s_1 \{ \mvar{n} \Def (\vcon{x}{n}) t_1 \} \eq s_2 \{ \mvar{n} \Def (\vcon{x}{n}) t_2 \} $$
is derivable from Second-Order Equational Logic. The categorical associativity and identity axioms making $\M(\E)$ a well-defined category  follow immediately, as do the facts that $\M(\E)$ comes equipped with the same cartesian structure as in $\M$ and that $(0)$ is exponentiable in $\M(\E)$. 

The category $\M$ classifies the most elementary second-order presentation $\E_0$, which has an empty set of operators and no equations. Indeed, $\M(\E_0)$ has morphisms tuples of terms (as the equivalence relation $\E_0$ singles out every term), and since all terms are elementary, $\M = \M(\E_0)$.

\begin{lem}
For a second-order equational presentation $\E$, the category $\M(\E)$ together with the canonical functor $M_\E \colon \M \To \M(\E)$ is a second-order algebraic theory.
\end{lem}

\begin{proof}
The functor $M_\E$ is the identity on objects and maps a tuple of terms $\langle t_1, \dots, t_n \rangle$ to the tuple of their equivalence classes $\big\langle [t_1]_\E, \dots, [t_n]_\E \big\rangle$. It preserves the cartesian and exponential structures of $\M$ as we have shown that they are, together with metasubstitution, respected by the equivalence relation~$\sim_\E$.
\end{proof}

We refer to $M_\E \colon \M \To \M(\E)$ as the second-order algebraic theory of $\E$. 

\begin{rem}
Consider a second-order signature $\Sigma$ and its induced second-order algebraic theory $M_\Sigma \colon \M \To \M(\Sigma)$. This construction is justified by considering a signature as just an equational presentation with an empty set of equations. Because of its universal property and the fact that every morphism of $\M$ decomposes as universal cartesian and exponential morphisms, it is clear that, since $M_\Sigma \colon \M \To \M(\Sigma)$ preserves the cartesian and exponential structure of $\M$, the algebraic theory $M_\Sigma$ is in this case simply an inclusion functor. 
\end{rem}

\textbf{The presentation of a theory.}
The \emph{internal language} $\IL(M)$ of a second-order algebraic theory $M \colon \M \To \Mlaw$ is the second-order equational presentation defined as follows: 
\begin{itemize}
\setlength\itemindent{32pt}
\item[(Operators)] For every $f \colon \tuple{m}{k} \To (n)$ in $\Mlaw$, we have an operator $\omega_f$ of arity $(m_1, \dots, m_k, 0^n)$, where $0^n$ stands for the appearance of $0$ $n$-times.
\item[(Equations)] Setting
$$ \term_f = \omega_f \big( (\ivcon{x}{m}{1}) \mvar{m}_1 \big[\ivcon{x}{m}{1}\big], \dots, (\ivcon{x}{m}{k}) \mvar{m}_k \big[\ivcon{x}{m}{k}\big], \vcon{x}{n} \big) $$
for every morphism $f \colon \tuple{m}{k} \To (n)$ in $\Mlaw$, we let $\IL(M)$ have equations
\begin{itemize}
\item[$(\E 1)$] $\mcon{m}{k} \sep \vcon{x}{n} \vdash s \eq \term_{M\langle s \rangle}$ \\
for every $\langle s \rangle \colon \tuple{m}{k} \To (n)$ in $\M$, and
\item[$(\E 2)$] $\mcon{m}{k} \sep \vcon{x}{n} \vdash \term_h \eq \term_g \{ \mvar{m}_i \Def (\ivcon{x}{n}{i}) \term_{f_i} \}_{i \in \card{l}}$ \\
for every \vspace{-0.5cm}
\begin{eqnarray*}
h &\colon& \tuple{m}{k} \To (n) \\
g &\colon& \tuple{n}{l} \To (n) \\
f_i &\colon& \tuple{m}{k} \To (n_i) \quad , \quad 1 \leq i \leq l
\end{eqnarray*}
such that $h = g \circ \langle f_1, \dots, f_l \rangle$ in $\Mlaw$.
\end{itemize}
\end{itemize}
We write $\Sigma(M)$ and $E(M)$ for these operators and equations, respectively.

\begin{rem}
This procedure of synthesising internal languages from second-order algebraic theories yields some redundancies in the resulting set of operators. For instance, the operator $\omega_f \colon (m_1, \dots, m_k, 0^n)$ induced by the morphism $f \colon \tuple{m}{k} \To (n)$ of $\Mlaw$ is essentially the same as the operator of same arity induced by the morphism $\ev_n \circ \big( f \times (0)^n \big) \colon (m_1, \dots, m_k, 0^n) \To (0)$. By \emph{essentially the same} we mean that the following is derivable from $(\E 1)$ and $(\E 2)$: 
$ \mcon{m}{k} \sep \vcon{x}{n} \vdash \term_f 
\eq  \term_{\ev_n \circ \big( f \times (0)^n \big)} $.
\end{rem}

\subsection{Towards second-order syntactic categorical algebraic theory correspondence}

Having presented the transformation between second-order algebraic theories and equational presentations, we now prove the first part of the mutual invertibility of these constructions. The second part of the proof requires the theory of second-order syntactic translations, and is hence postponed to the next Section.

\begin{thm}[Theory/presentation correspondence]
Every second-order algebraic theory $M \colon \M \To \Mlaw$ is isomorphic to the second-order algebraic theory $M_{\IL(M)} \colon \M \To \M(\IL(M))$ of its associated second-order equational presentation.
\end{thm}

\begin{proof} 
See Appendix B.
\end{proof}

\section{Second-Order Syntactic Translations}
\label{so_trans}

Algebraic theories come with an associated notion of algebraic translation, their morphisms. In the second-order universe, the syntactic morphism counterpart has yet to be formalised. 
To this end, we distill a notion of second-order syntactic translation between second-order equational presentations that corresponds to the canonical notion of morphism between second-order algebraic theories. These syntactic translations provide a mathematical formalisation of notions such as encodings and transforms. The correctness of our definition is established by showing a categorical equivalence between algebraic and syntactic translations. This completes the \emph{Second-Order Syntactic Categorical Algebraic Theory Correspondence}, by which second-order algebraic theories and their algebraic translations correspond to second-order equational presentations and their syntactic translations.

\subsection{Second-Order Signature Translations}

A \emph{syntactic translation} $\trans \colon \Sigma \To \Sigma '$ between second-order signatures is given by a mapping from the operators of $\Sigma$ to the terms of $\Sigma '$ as follows:
$$ \omega \colon \tuple{m}{k} \quad \mapsto \quad \mcon{m}{k} \Sep \emptycon \vdash \trans_\omega $$
Note that the term associated to an operator has an empty variable context and that the metavariable context is determined by the arity of the operator. 
A second-order syntactic translation $\trans \colon \Sigma \To \Sigma '$ extends to a mapping from the terms of $\Sigma$ to the terms of $\Sigma '$
\begin{eqnarray*}
T_\Sigma & \To & T_{\Sigma '} \\
\Theta \Sep \Gamma \vdash t & \mapsto & \Theta \Sep \Gamma \vdash \trans(t)
\end{eqnarray*}
according to the following definition by induction on term structure:
\begin{itemize}
\item[-] $\trans(x) = x$
\item[-] $\trans \big(\mvar{m}[t_1, \dots, t_m] \big) = \mvar{m} \big[ \trans(t_1), \dots, \trans(t_m) \big]$
\item[-] $\trans \big( \omega \big( (\ivcon{x}{n}{1})t_1, \dots (\ivcon{x}{n}{k})t_k  \big) \big) = \trans_\omega \big\{ \mvar{m}_i \Def (\ivcon{x}{n}{i}) \trans(t_i) \big\}_{i \in \card{k}}$
\end{itemize}

We refer to this mapping as the \emph{translation extension} or the induced \emph{translation of terms}. 
Substituting for variables in a term followed by syntactic translation of the resulting term amounts to the same as term translation followed by substitution, and similarly for metasubstitution. This subtlety is crucial when defining morphisms of signatures as syntactic translations. We leave the simple syntactic manipulation required to prove the following result to the reader.

\begin{lem}[Compositionality]
The extension of a syntactic translation between second-order signatures commutes with substitution and metasubstitution.
\end{lem}

\pagebreak

\begin{ex}
\begin{itemize}
\item[]
\item[(1)] The simplest way to translate a second-order signature is to map it to itself. Every operator can be mapped to the $\textquoteleft$simplest' term induced by that operator. More formally, for $\Sigma$ a second-order signature, the mapping
\begin{eqnarray*}
\omega \colon n_1, \dots, n_l & \mapsto & \mcon{n}{l} \Sep \emptycon \vdash \omega \big( (\vvar{y}_1) \mvar{n}_1 [\vvar{y}_1], \dots, (\vvar{y}_n) \mvar{n}_l [\vvar{y}_l] \big)
\end{eqnarray*}
defines a second-order syntactic translation. We will later show that this defines the \emph{identity} syntactic translation.
\item[(2)] It is well-known that the basic mono-sorted $\lambda$-calculus may be used to model simple arithmetic structures and operations. For instance, \emph{Church numerals} are a way of formalising natural numbers via the $\lambda$-calculus. The Church numeral $n$ is roughly a function which takes a function $f$ as argument and returns the $n$-th composition of $f$. The encoding of basic operations on natural numbers, such as addition and multiplication, via Church's $\lambda$-calculus can be formalised as a syntactic translation as follows:
\begin{eqnarray*}
\tt{add} \colon (0,0) & \mapsto & \mvar{m} \colon [0], \mvar{n} \colon [0] \Sep \emptycon \vdash \lambda fx.\mvar{m} f (\mvar{n} f x) \\
\tt{mult} \colon (0,0) & \mapsto & \mvar{m} \colon [0], \mvar{n} \colon [0] \Sep \emptycon \vdash \lambda f.\mvar{m} (\mvar{n} f)
\end{eqnarray*}
\end{itemize}
\end{ex}

\subsection{Second-Order Equational Translations}

A \emph{syntactic translation} $\trans \colon \E \To \E '$ between second-order equational presentations $\E = (\Sigma,E)$ and $\E '=(\Sigma ',E')$ is a signature translation which preserves the equational theory of $\E$ in the sense that axioms are mapped to theorems. Formally, it is a syntactic translation $\trans \colon \Sigma \To \Sigma '$ such that, for every axiom $\Theta \Sep \Gamma \vdash_\E s \eq t$ in $E$, the judgement $\Theta \Sep \Gamma \vdash_{\E '} \trans (s) \eq \trans (t)$ is derivable from $E'$. 
The condition that only axioms are required to be mapped to theorems is strong enough to ensure that all theorems of $\E$ are mapped to theorems of $\E '$.

\begin{lem}
The extension of a syntactic translation between second-order equational presentations preserves second-order equational derivability.
\end{lem}

\begin{proof}
One needs to only check the extended metasubstitution derivation rule of Second-Order Equational Logic (Figure 1). Indeed, having
$$ \mcon{m}{k} \Sep \Gamma \vdash_{\E '} \trans(s) \eq \trans(t) \quad \textrm{and} \quad \Theta \Sep \Gamma ', \ivcon{x}{m}{i} \vdash_{\E '} \trans(s_i) \eq \trans(t_i) $$
for $1 \leq i \leq k$ implies
$$ \Theta \Sep \Gamma, \Gamma ' \vdash_{\E '} \trans(s) \big\{ \mvar{m}_i \Def (\ivcon{x}{m}{i}) \trans(s_i) \big\}_{i \in \card{k}} \eq \trans(t) \big\{ \mvar{m}_i \Def (\ivcon{x}{m}{i}) \trans(t_i) \big\}_{i \in \card{k}} $$
by extended metasubstitution, which, by the Compositionality Lemma (Lemma ?), further gives
$$ \Theta \Sep \Gamma , \Gamma ' \vdash_{\E '} \trans \big( s \big\{ \mvar{m}_i \Def (\ivcon{x}{m}{i}) \trans(s_i) \big\}_{i \in \card{k}} \big) \eq \trans \big( t \big\{ \mvar{m}_i \Def (\ivcon{x}{m}{i}) \trans(s_i) \big\}_{i \in \card{k}} \big) . $$
\end{proof}

\textbf{Syntactic translation composition.} The \emph{composite} of equational translations $\trans \colon \E_1 \To \E_2$ and $\trans ' \colon \E_2 \To \E_3$ is the translation $(\trans ' \circ \trans) \colon \E_1 \To \E_3$ defined by mapping an operator $\omega$ of $\E_1$ to the term $\trans '(\trans_\omega)$ of $\E_3$. Its extension on a term $t$ is simply $\trans '(\trans(t))$,
which can be easily verified by structural induction.
Because $\trans$ and $\trans '$ preserve equational derivability, the equation $\Theta \Sep \Gamma \vdash \trans ' \big( \trans(s)\big) \eq \trans '\big( \trans(t)\big)$ is a theorem of $\E_3$ whenever $\Theta \Sep \Gamma \vdash s \eq t$ is an axiom of $\E_1$, and thus, the composite $(\trans' \circ \trans)$ is an equational translation. 
Note that composition of equational translations is an associative operation: 
$\big( (\trans'' \circ \trans') \circ \trans \big)(\omega) \squad = \squad (\trans'' \circ \trans') (\trans_\omega) \squad = \squad \trans'' \big( \trans'(\trans_\omega) \big) \squad = \squad \trans'' \big( (\trans'\circ\trans) (\omega) \big) \squad = \squad \big( \trans'' \circ (\trans'\circ\trans) \big) (\omega) $,
where of course all composites are assumed to be well-defined. \\

\textbf{The syntactic identity translation.} For a second-order equational presentations $\E$, the \emph{syntactic identity translation} $\trans^\E \colon \E \To \E$ is defined by mapping an operator $ \omega \colon \tuple{m}{k} $ to the term
\begin{equation*}
\begin{multlined}
\mcon{m}{k} \Sep \emptycon \vdash \\ \omega \big( (\ivcon{x}{m}{1}) \mvar{m}_1 [\ivcon{x}{m}{1}], \dots, (\ivcon{x}{m}{k}) \mvar{m}_1 [\ivcon{x}{m}{k}] \big) \quad . 
\end{multlined}
\end{equation*}
The extension of $\trans^\E$ on terms is just the identity mapping, which is again easily verified by structural induction. This immediately implies that an axiom $\Theta \Sep \Gamma \vdash_\E s \eq t $ is mapped to itself under $\trans^\E$, making it an equational translation.

Note that $\trans^\E$ is indeed the identity in the space of equational translations and their composition, since for any $\trans \colon \E_1 \To \E_2$ and $\omega \colon \tuple{m}{k}$ an operator of $\E_1$, we have
\begin{eqnarray*}
\trans^{\E_2} \big( \trans (\omega_1) \big) & = & \trans(\omega_1) \quad , \\
\textrm{and}\hspace{100pt} & \quad & \\
\trans \big( \trans^{\E_1} (\omega) \big) & = & \trans \big( \omega \big( \dots, (\ivcon{x}{m}{i}) \mvar{m}_i [\ivcon{x}{m}{i}] ,\dots \big) \big) \\
& = & \trans_\omega \big\{ \mvar{m}_i \Def (\ivcon{x}{m}{i}) \trans \big( \mvar{m}_i [\ivcon{x}{m}{i}] \big) \big\}_{i \in \card{k}} \\
& = & \trans_\omega \big\{ \mvar{m}_i \Def (\ivcon{x}{m}{i})  \mvar{m}_i [\ivcon{x}{m}{i}] \big\}_{i \in \card{k}} \\
& = & \trans_\omega \quad .
\end{eqnarray*}

\textbf{The category of second-order equational presentations.}  We denote by $\SOEP$ the category of second-order equational presentations and second-order syntactic translations. The previous discussion surrounding composition and identity ensures that this is a well-defined category.

\subsection{Syntactic and Algebraic Translations}

A syntactic translation $\trans \colon \E \To \E '$  of second-order equational presentations $\E = (\Sigma,E)$ and $\E ' = (\Sigma ',E')$ induces the algebraic translation
$$ \M(\trans) \colon \M(\E) \To \M(\E ') $$
mapping $\big\langle [t_1]_\E, \dots, [t_l]_\E \big\rangle$ to $\big\langle [\trans(t_1)]_{\E '}, \dots, [\trans(t_l)]_{\E '} \big\rangle$. Note that the induced algebraic translation $\M(\trans)$ is essentially specified by the extension of the syntactic translation $\trans$ on terms. This definition respects equivalence since the extension of $\trans$ preserves equational derivability, and thus $\Theta \Sep \Gamma \vdash_\E s \eq t$ implies $\Theta \Sep \Gamma \vdash_{\E '} \trans(s) \eq \trans(t)$. From the Compositionality Lemma (Lemma 6.1), we know that extensions of syntactic translations commute with substitution and metasubstitution, which yields functoriality of $\M(\trans)$. Since translation extensions act as the identity on elementary terms, the functor $\M(\trans)$ commutes with the theories $M_\E \colon \M \To \M(\E)$ and $M_{\E '} \colon \M \To \M(\E ')$, making it indeed an algebraic translation. 

This development gives a functor 
\begin{eqnarray*}
\M(-) \quad \colon \quad \SOEP &\To& \SOAT \\
\E & \mapsto & M_\E \colon \M \To \M(\E) \\
\trans \colon \E \To \E ' & \mapsto & \M(\trans) \colon \M(\E) \To \M(\E ')
\end{eqnarray*}
mapping an equational presentation to its classifying theory, and a syntactic translation to its induced algebraic translation. Since the extension of the syntactic identity translation $\trans^\E \colon \E \To \E$ is the identity on terms, it is mapped under $\M(-)$ to the identity algebraic translation $\M(\trans^\E)$ mapping $\big\langle \dots,[t]_\E ,\dots \big\rangle$ to itself. Also, given syntactic translations $\trans \colon \E_1 \To \E_2$ and $\trans ' \colon \E_2 \To \E_3$, we have 
$$ \M(\trans ' \circ \trans)\big( [t]_{\E_1} \big) = \big[ (\trans ' \circ \trans)(t) \big]_{\E_3} = \big[ \trans ' \big( \trans(t) \big) \big]_{\E_3} = \M(\trans ') \big( \big[ \trans(t) \big]_{\E_2} \big) = \big( \M(\trans ') \circ \M(\trans) \big) \big( [t]_{\E_1} \big) \quad , $$
which establishes functoriality of $\M(-)$. \\

In the other direction, an algebraic translation $F \colon \Mlaw \To \Mlaw '$ between second-order algebraic theories $M \colon \M \To \Mlaw$ and $M' \colon \M \To \Mlaw '$ induces the syntactic translation 
$$ \IL(F) \colon \IL(M) \To \IL(M') \quad , $$
which, for a morphism $f \colon \tuple{m}{k} \To (n)$ of $\Mlaw$, maps the operator $\omega_f$ of $\IL(M)$ to the term
$$ \mcon{m}{k}, \mvar{n}_1 \colon [0], \dots, \mvar{n}_n \colon [0] \Sep \emptycon \vdash \term_{Ff} \big\{ x_i \Def \mvar{n}_i[] \big\}_{i \in \card{n}} \quad , $$
where we recall that 
$$ \term_{Ff} = \omega_{Ff} \big( (\ivcon{x}{m}{1}) \mvar{m}_1 \big[\ivcon{x}{m}{1}\big], \dots, (\ivcon{x}{m}{k}) \mvar{m}_k \big[\ivcon{x}{m}{k}\big], \vcon{x}{n} \big) \quad . $$

We verify that $\IL(F)$ is indeed an equational translation by looking at the induced translations on the terms of the left- and right-hand side of the axioms of $\IL(M)$. Recall from Section 5.1 that these axioms are given by $(\E 1)$ and $(\E 2)$. Consider $(\E 1)$, which states that for $\langle s \rangle \colon \tuple{m}{k} \To (n)$ of $\M$, we have the equation $\mcon{m}{k} \Sep \vcon{x}{n} \vdash s \eq \term_{M\langle s \rangle}$ in $\IL(M)$. Since $s$ is elementary, its image under the translation $\IL(F)$ is also given by $M\langle s \rangle$. On the other hand, note that $\IL(f)(\term_{M\langle s \rangle}) = \term_{(F \circ M) \langle s \rangle} = \term_{M' \langle s \rangle}$. From the axiom $(\E 1)$ of $\IL(M')$, we have that $s \eq \term_{M \langle s \rangle}$, and therefore
$ \mcon{m}{k} \Sep \vcon{x}{n} \vdash \IL(F)(s) \eq \IL(F)(\term_{M \langle s \rangle}) $
in $\IL(M')$. Similarly, for the axiom $(\E 2)$ of $\IL(M)$, and in the notation of Section 5.1, we have that $\IL(F)(\term_h) = \term_{Fh}$, and on the other hand:
\begin{align*}
& \quad \IL(F) \Big( \term_g \big\{ \mvar{m}_i \Def (\vvar{x_i}) \term_{f_i} \big\}_{i \in \card{l}} \Big) \\
= & \quad \IL(F) \big( \term_g \big) \big\{ \mvar{m}_i \Def (\vvar{x_i}) \IL(F)\big(\term_{f_i}\big) \big\}_{i \in \card{l}}  \\
= &\quad \term_{Fg} \big\{ \mvar{m}_i \Def (\vvar{x_i}) \term_{Ff_i} \big\}_{i \in \card{l}} \quad .
\end{align*}
Hence, the image of axiom $(\E 2)$ of $\IL(M)$ under the translation $\IL(F)$ is just axiom $(\E 2)$ of $\IL(M')$. This makes $\IL(F)$ indeed an equational translation. 

We have essentially defined the functor
\begin{eqnarray*}
\IL(-)\quad \colon \quad \SOAT & \To & \SOEP \\
M \colon \M \To \Mlaw & \mapsto & \IL(M) \\
F \colon \Mlaw \To \Mlaw ' & \mapsto & \IL(F) \colon \IL(M) \To \IL(M')
\end{eqnarray*}
mapping a second-order algebraic theory to its internal language, and an algebraic translation to its induced syntactic translation. 
Some simple syntactic manipulation is needed to show that $\IL(-)$ is functorial. Given a morphism $f \colon \tuple{m}{k} \To (n)$ in $\Mlaw$, the translation $\IL(id_\Mlaw)$ maps the operator $\omega_f$ of $\IL(M)$ to the term
$$ \mcon{m}{k}, \mvar{n}_1 \colon [], \dots, \mvar{n}_n \colon [] \Sep \emptycon \vdash \term_f \big\{ x_i \Def \mvar{n}_i[] \big\}_{i \in \card{n}} \quad ,$$
which is the image of $\omega_f$ under the syntactic identity translation $\trans^{\IL(M)}$. Moreover, given algebraic translations $F \colon \Mlaw_1 \To \Mlaw_2$ and $G \colon \Mlaw_2 \To \Mlaw_3$ and a morphism $g \colon \tuple{n}{j} \To (l)$, the image of $\omega_g$ of $\IL(M_1)$ under the composite translation $\IL(G) \circ \IL(F)$ is given by the term
\begin{eqnarray*}
 \mcon{n}{j}, \mvar{l}_1 \colon [], \dots, \mvar{l}_l \colon [] \Sep \emptycon & \vdash & \IL(G) \big( \term_{Fg} \big\{ x_i \Def \mvar{l}_i [] \big\}_{i \in \card{l}} \big) \\
& = & \IL(G) (\omega_{Fg}) \\
& = & \term_{(G \circ F)(g)}\big\{ x_i \Def \mvar{l}_i [] \big\}_{i \in \card{l}} \\
& = & \IL(G \circ F) (\omega_g) \quad .
\end{eqnarray*}

\subsection{Second-Order Syntactic Categorical Algebraic Theory Correspondence}

Second-order syntactic translations embody the mathematical machinery that enables us to compare second-order equational presentations at the syntactic level without having to revert to their categorical counterparts. In particular, the question of when two presentations are \emph{essentially the same} can now be answered via the notion of syntactic isomorphism. A second-order syntactic translation $\trans \colon \E \To \E '$ is an \emph{isomorphism}, if it has an inverse $\bar{\trans}$ yielding the syntactic identity translation on $\E$ (respectively $\E '$) when composed to the left (respectively right) with $\trans$. 
This is used to show the second direction of the invertibility of constructing theories from presentations, and vice versa. More precisely, we prove that every second-order equational presentation is isomorphic to the second-order equational presentation of its associated algebraic theory. 

Keeping this objective in mind, define, for a given second-order equational presentation $\E$ with classifying algebraic theory $M_\E \colon \M \To \M(\E)$, the \emph{natural translation}
$ \tnat_\E \colon \E \To \IL(M_\E) $
by mapping an operator $\omega \colon \tuple{m}{k}$ of $\E$ to the term 
$ \mcon{m}{k} \Sep \emptycon \vdash \term_{\big\langle[\trans^\E_\omega]_\E \big\rangle}$, 
where we remind the reader that $\trans^\E(\omega)$ is the image of $\omega$ under the identity translation $\trans^\E$, and hence $\big\langle[\trans^\E(\omega)]_\E \big\rangle \colon \tuple{m}{k} \To (0)$ is a morphism of $\M(\E)$. 
The fact that the natural translation $\tnat_\E$ is an equational translation relies on the following special property of its extension on terms.

\begin{lem}
For any second-order equational presentation $\E$, the extension of the natural translation $\tnat_\E \colon \E \To \IL(M_\E)$ on a term
$$ \mcon{m}{k} \Sep \vcon{x}{n} \vdash s $$
of $\E$ is given by the term
$$ \mcon{m}{k} \Sep \vcon{x}{n} \vdash \term_{\langle [s]_\E \rangle} $$
of $\IL(M_\E)$.
\end{lem}

Given an axiom $\mcon{m}{k} \Sep \vcon{x}{n} \vdash t \eq t'$ of $\E$ then, the operators $\omega_{\langle [t]_\E \rangle}$ and $\omega_{\langle [t']_\E \rangle}$ are obviously equal, which makes the terms $\term_{\langle [t]_\E \rangle}$ and $\term_{\langle [t']_\E \rangle}$ of $\IL(M_\E)$ syntactically equal. This implies the equational derivability of
$$ \mcon{m}{k} \Sep \vcon{x}{n} \vdash_{\IL(M_\E)} \term_{\langle [t]_\E \rangle} \eq \term_{\langle [t']_\E \rangle} \quad , $$
which, together with Lemma 6.4, yields
$$ \mcon{m}{k} \Sep \vcon{x}{n} \vdash_{\IL(M_\E)} \tnat_\E(t) \eq \tnat_\E (t') \quad , $$
making $\tnat_\E$ indeed an equational translation. 

In the other direction, define the \emph{opposite natural translation} 
$\tnatbar_\E \colon \IL(M_\E) \To \E $ 
by mapping, for a morphism $\langle [t]_\E \rangle \colon \tuple{m}{k} \To (n)$ of $\M(\E)$, the operator $\omega_{\langle [t]_\E \rangle} \colon (m_1, \dots, m_k, 0^n)$ to 
$$ \mcon{m}{k}, \mvar{n}_1 \colon [0], \dots, \mvar{n}_n \colon [0] \Sep \emptycon \vdash t \big\{ x_i \Def \mvar{n}_i[] \big\}_{i \in \card{n}} \quad .  $$

This mapping is well-defined in the sense that it respects the equivalence with respect to $\E$, as from Second-Order Equational Logic we know that the operation of substitution in extended metavariable context is a congruence.  
To verify that, according to this definition, $\tnatbar_\E$ is really an equational translation, one needs to show that the two axioms $(\E 1)$ and $(\E 2)$ of $\IL(M_\E)$ are mapped under $\tnatbar_\E$ to theorems of $\E$. A similar argument to the verification of the preservation of equations of an induced syntactic translation can be used, and so we skip over the details here. 

\begin{thm}[Second-order presentation/theory correspondence]
Every second-order equational presentation $\E$ is isomorphic to the second-order equational presentation $\IL(M_\E)$ of its associated algebraic theory $M_\E \colon \M \To \M(\E)$.
\end{thm}

\begin{proof}
The isomorphism is witnessed by the natural translation $\tnat_\E \colon \E \To \IL(M_\E)$ with its inverse given by the opposite natural translation $\tnatbar_\E \colon \IL(M_\E)$. Indeed, an operator $\omega \colon \tuple{m}{k}$ of $\E$ is mapped under the composite $\tnatbar_\E \circ \tnat_\E$ to 
$$ \mcon{m}{k} \Sep \emptycon  \vdash  \tnatbar_\E \big( \omega_{\langle [ \trans^\E (\omega) ]_\E \rangle} \big) \squad = \squad \trans^\E (\omega) \quad . $$
In the other direction, for a morphism $\langle [s]_\E \rangle \colon \tuple{m}{k} \To (n)$ of $\M(\E)$, the operator $\omega_{\langle [s]_\E \rangle}$ is mapped under $\tnat_\E \circ \tnatbar_\E$ to 
\begin{eqnarray*}
\mcon{m}{k}, \mvar{n}_1 \colon [0], \dots, \mvar{n}_n \colon [0] \Sep \emptycon & \vdash & \tnat_\E \big( s \big\{ x_i \Def \mvar{n}_i[] \big\}_{i \in \card{n}} \big) \\
& = & \tnat_\E (s) \big\{ x_i \Def \tnat_\E (\mvar{n}_i[]) \big\}_{i \in \card{n}} \\
& = & \tnat_\E (s) \big\{ x_i \Def \mvar{n}_i[] \big\}_{i \in \card{n}} \\
& = & \term_{\langle [s]_\E \rangle} \big\{ x_i \Def \mvar{n}_i[] \big\}_{i \in \card{n}} \\
& = & \trans^{\IL(M_\E)} (\omega_{\langle [s]_\E \rangle}) \quad . 
\end{eqnarray*}
\end{proof}

We finally illustrate that the constructions of induced algebraic and syntactic translations are mutually inverse, thereby establishing them as the correct notions of morphisms of, respectively, algebraic theories and equational presentations. 

\begin{thm}[Second-Order Syntactic Categorical Algebraic Theory Correspondence]
The categories $\SOAT$ and $\SOEP$ are equivalent.
\end{thm}

\begin{proof}
The equivalence is given by the functors 
$$\IL(-) \colon \SOAT \To \SOEP \qquad \mathrm{and} \qquad \M(-) \colon \SOEP \To \SOAT$$
together with the natural transformation $\nat \colon \textrm{Id}_\SOAT \To \M(\IL(-))$ with component at a second-order algebraic theory $M \colon \M \To \Mlaw$ given by the isomorphism
$$ \nat_M \colon \Mlaw \To \M(\IL(M)) $$
defining the Theory/Presentation Correspondence of Theorem 5.5, and also the natural transformation \linebreak$\tnat \colon \textrm{Id}_\SOEP \To \IL(\M(-))$ with component at a second-order equational presentation $\E = (\Sigma,E)$ given by the isomorphism
$$ \tnat_\E \colon \E \To \IL(M_\E) $$
defining the Presentation/Theory Correspondence of Theorem 6.5. From the very definitions of the functors $\M(-)$ and $\IL(-)$ and the isomorphisms $\nat_{(-)}$ and $\tnat_{(-)}$, the diagrams
\begin{diagram}[small]
\Mlaw & \rTo^{F} & \Mlaw ' & \qquad & \E & \rTo^{\trans} & \E ' \\
\dTo^{\nat_M} & & \dTo_{\nat_{M'}} & \qquad & \dTo^{\tnat_\E} & &  \dTo_{\tnat_{\E '}} \\
\M(\IL(M)) & \rTo^{\M(\IL(F))} & \M(\IL(M')) & \qquad & \IL(M_\E) & \rTo^{\IL(\M(\trans))} & \IL(M_{\E '}) \\
\end{diagram}

commute for any second-order algebraic translation $F$ between algebraic theories $M \colon \M \To \Mlaw$ and $M' \colon \M \To \Mlaw '$, and any second-order syntactic translation $\trans \colon \E \To \E '$ of equational presentations $\E$ and $\E '$, thereby establishing naturality of $\nat$ and $\tnat$. 

Now, consider the diagram above on the left; its commutativity is trivial on the objects of $\Mlaw$. Given a morphism $f \colon \tuple{m}{k} \To (n)$ of $\Mlaw$, its image under $\nat_{M'} \circ F$ is the morphism
$ \big\langle \big[ \term_{Ff} \big]_{\IL(M)} \big\rangle \colon \tuple{m}{k} \To (n) $.
Going the other way, the image of $f$ under $\M(\IL(F)) \circ \nat_M$ is given by
\begin{align*}
& \quad \quad  \M(\IL(F)) \big\langle \big[ \term_f \big]_{\IL(M)} \big\rangle \\
& = \quad \big\langle \big[ \IL(F) \big( \term_{f} \big) \big]_{\IL(M)} \big\rangle \\
& = \quad  \big\langle \big[ \IL(F) (\omega_f) \{ \mvar{n}_i \Def x_i \}_{i \in \card{n}} \big]_{\IL(M)} \big\rangle \\
& = \quad \big\langle \big[ \term_{Ff} \big]_{\IL(M)} \big\rangle \quad .
\end{align*}
To verify the commutativity of the diagram to the right, note that the image of an operator $\omega \colon n_1, \dots, n_l$ of $\E$ under the composite $\tnat_{\E '} \circ \trans$ is the term
$ \mcon{n}{l} \Sep \emptycon \vdash \term_{\langle [\trans(\omega)]_{\E '} \rangle}$. 
On the other hand, the image of $\omega$ under $\IL(\M(\trans)) \circ \tnat_\E$ is given by
\begin{align*}
& \quad \quad \IL(\M(\trans)) \big( \term_{\langle [ \term_\omega ]_\E \rangle} \big) \\
& = \quad \term_{\M(\trans)\langle [ \term_\omega ]_\E \rangle} \\
& = \quad \term_{\langle [ \trans(\term_\omega) ]_{\E '} \rangle} \\
& = \quad \term_{\langle [\trans(\omega)]_{\E '} \rangle}  \quad .
\end{align*}
\end{proof}

\section{Second-Order Functorial Semantics}
\label{funsem}

Before generalising Lawvere's functorial semantics to the second order, we need recall and develop some aspects of the theory of clones.
In modern first-order universal algebra \cite{Cohn1965}, clones provide an abstract presentation of algebras. 

\subsection{Clone Structures}

One understands by a \emph{clone} on a set $S$ the set of all \emph{elementary operations} on $S$, which includes projections $S^n \To S$ for any $n \in \N$ and is closed under multiple finitary function composition. 
More formally, a \emph{clone} in a cartesian category $\cat{C}$ is an $\N$-indexed collection $\{ C_n \}_{n \in \N}$ of objects of $\cat{C}$ equipped with \emph{variable maps} $\iota^{(n)}_i \colon 1 \To C_n$, $(i \in \card{n})$, for each $n \in \N$, and \emph{substitution maps} $\varsigma_{m,n} \colon C_m \times (C_n)^m \To C_n$ for each $m,n \in \N$, such that the following commute:
\begin{diagram}[small]
C_n \times 1 && \rTo^{id_{C_n} \times \langle \iota^{(n)}_1, \dots, \iota^{(n)}_n \rangle} & & C_n \times (C_n)^n \\
& \rdTo_{\pi_1} & & \ldTo_{\varsigma_{n,n}} \\
& & C_n 
\end{diagram}
\begin{diagram}[small]
1 \times (C_n)^m & & \rTo^{\pi_2} & & (C_n)^m \\
\dTo^{\iota^{(m)}_i \times id_{(C_n)^m}} & & & & \dTo_{\pi_i} \\
C_m \times (C_n)^m & & \rTo^{\varsigma_{m,n}} & & C_n 
\end{diagram}
\begin{diagram}[small]
C_l \times (C_m)^l \times (C_n)^m & & \rTo^{\varsigma_{l,m} \times id_{(C_n)^m}} & & C_m \times (C_n)^m \\
\dTo^{\varphi} & & & & \dTo_{\varsigma_{m,n}} \\
C_l \times (C_n)^l & & \rTo^{\varsigma_{l,n}} & & C_n \\ 
\end{diagram}

where $\varphi$ is the morphism $id_{C_l} \times \langle \varsigma_{m,n} \circ (\pi_i \times id_{(C_n)^m}) \rangle_{i \in \card{l}}$. 
Every clone $\{C_n\}_{n\in\N}$ in $\cat{C}$ canonically extends to a functor $\ff \To \cat{C}$ defined by mapping $n$ to $C_n$. Moreover, given another cartesian category $\cat{D}$, any cartesian functor $F \colon \cat{C} \To \cat{D}$ preserves the clone structure in $\cat{C}$, in the sense that every clone $\{C_n\}_{n \in \N}$ of $\cat{C}$ induces the clone $\{ F(C_n)\}_{n \in \N}$ with structure maps given by $F(\iota^{(n)}_i)$ and $F(\varsigma_{m,n} \circ \cong)$ (for $m,n \in \N$ and $i \in \card{n}$), where $\cong$ is the canonical isomorphism $F(C_m) \times \big(F(C_n)\big)^m \To F\big( C_m \times (C_n)^m \big)$. 

Given a cartesian category $\cat{C}$, the category $\Clone(\cat{C})$ has objects clones $\{ C_n \}_{n \in \N}$ of $\cat{C}$. A \emph{clone homomorphism} $ \{ C_n \}_{n \in \N} \To \{ D_n \}_{n \in \N} $ is an $\N$-indexed family of morphisms $ \{ h_n \colon C_n \To D_n \}_{n \in \N} $ of $\cat{C}$ such that for all $m,n \in \N$ the following commute:
\begin{diagram}
1 & \rTo^{\iota^{(C)}_i} & C_n & & & C_m \times (C_n)^m & \rTo^{\varsigma^{(C)}_{m,n}} & C_n \\
& \rdTo^{\iota^{(D)}_i} & \dTo_{h_n} & & & \dTo^{h_m \times (h_n)^m} & & \dTo_{h_n} \\
& & D_n & & & D_m \times (D_n)^m & \rTo^{\varsigma^{(D)}_{m,n}} & D_n
\end{diagram}

\subsection{Clone semantics}

A clone for a second-order signature $\Sigma$ in a cartesian category $\cat{C}$
is given by a clone $\{ S_n \}_{n \in \N}$ in $\cat{C}$, together with, for each $n \in \N$, \emph{natural operator maps}
$ \tilde{\omega}_n \colon S_{n+n_1} \times \cdots \times S_{n+n_l} \To S_n $
for every operator $\omega \colon n_1, \dots, n_l$ of $\Sigma$, such that, for all $n,m \in \N$, the diagram
\begin{diagram}[small]
& & \prod_{i \in \card{l}}  S_{n+n_i} \times (S_{m+n_i})^{n+n_i} & & \\
{\scriptstyle \langle \textrm{id} \times \upsilon_{n_i} \rangle_{i \in \card{l}}} &\ruTo^{}  & & \rdTo^{\qquad} &  {\scriptstyle\prod_{i \in \card{l}} \varsigma_{n+n_i,m+n_i}}  \\
\prod_{i \in \card{l}} S_{n+n_i} \times (S_m)^n & & & & \prod_{i \in \card{l}} S_{m+n_i} \\
\dTo^{\scriptstyle \tilde{\omega}_n \times \upsilon_0} & & & & \dTo_{\scriptstyle \tilde{\omega}_m} \\
S_{n} \times (S_{m})^{n} & & \rTo^{\scriptstyle \varsigma_{n,m}} & & S_{m} \\
\end{diagram}
commutes, where for each $k \in \N$, the morphism $\upsilon_k$ is given by
\begin{diagram}[small]
(S_m)^n \squad \cong \squad (S_m)^n \times 1 & \rTo^{(S_j)^n \times \langle \iota^{(m+k)}_{m+i} \rangle_{i \in \card{k}}} & (S_{m+k})^n \times (S_{m+k})^k \squad \cong \squad (S_{m+k})^{n+k} \quad ,
\end{diagram}
and $j$ is the inclusion $\card{m} \hookrightarrow \card{m+k}$. Note that at $0$, $\upsilon_0$ is just the identity on $(S_m)^n$. 
The naturality condition on the operator maps above refers to the canonical action for any $f \colon m \To n$ in $\ff$ given by the composite
\begin{diagram}[small]
C_m \squad \cong \squad C_m \times 1 & \rTo^{C_m \times \langle \iota^{(n)}_{f1}, \dots, \iota^{(n)}_{fm} \rangle} & C_m \times (C_n)^m & \rTo^{\varsigma_{m,n}} & C_n 
\end{diagram}
that is available in any clone. 

We write $\Sigma\mbox{-}\Clone(\cat{C})$ for the category of $\Sigma$-clones in $\cat{C}$, with morphisms given by clone homomorphisms which commute with the natural operator maps $\tilde{\omega}_n$ for every operator $\omega$ of $\Sigma$ and $n \in \N$. 
A $\Sigma$-clone $\{S_n\}_{n \in \N}$ in a cartesian category $\cat{C}$ is \emph{preserved} under a functor $F \colon \cat{C} \To \cat{D}$ if $\{ F(S_n) \}_{n \in \N}$ is a $\Sigma$-clone in the cartesian category $\cat{D}$ with structure maps given by the image under $F$ of the structure maps associated to the clone $\{S_n\}_{n \in \N}$. It is evident that clones are necessarily preserved under cartesian functors.

A $\Sigma$-clone $\{ S_n \}_{n \in \N}$ in $\cat{C}$ induces an interpretation of second-order terms in $\cat{C}$. For the metavariable context $\Theta = (\mcon{m}{k})$ and variable context $\Gamma = (\vcon{x}{n})$, the interpretation of a term $\Theta \Sep \Gamma \vdash t$ under the clone $\{ S_n \}_{n \in \N}$ is a morphism 
$ \alg{\Theta \Sep \Gamma \vdash t}{S} \colon \prod_{i \in \card{k}} S_{m_i} \To S_n $
given by induction on the structure of the term $t$ as follows:
\begin{itemize}
\item[-] $\alg{\Theta \Sep \Gamma \vdash x_i}{S}$ is the composite
\begin{diagram}[small]
\prod_{i \in \card{k}} S_{m_i} & \rTo^{!} & 1 & \rTo^{\iota^{(n)}_i} & S_n \quad .
\end{diagram}
\item[-] $\alg{\Theta \Sep \Gamma \vdash \mvar{m}_i[t_1, \dots, t_{m_i}]}{S}$ is the composite
\begin{diagram}[small]
\prod_{i \in \card{k}} S_{m_i} & \rTo^{\langle \pi_i, \alg{\Theta \Sep \Gamma \vdash t_1}{S}, \dots, \alg{\Theta \Sep \Gamma \vdash t_{m_i}}{S} \rangle} & S_{m_i} \times (S_n)^{m_i} & \rTo^{\varsigma_{m_i,n}} & S_n \quad .
\end{diagram} 
\item[-] For an operator $\omega \colon n_1, \dots, n_l$, $\alg{\Theta \Sep \gamma \vdash \omega\big( (\vvar{y_1})t_1, \dots, (\vvar{y_l})t_l\big)}{S}$ is the composite
\begin{diagram}[small]
\prod_{i \in \card{k}} S_{m_i} & \rTo^{\langle \alg{\Theta \Sep \Gamma_{n_i} \vdash t_i}{S} \rangle_{i \in \card{l}}} & \prod_{i \in \card{l}} S_{n+n_i} & \rTo^{\tilde{\omega}} & S_n \quad ,
\end{diagram}
where for $i \in \card{l}$, $\Gamma_{n_i}$ is the context $\Gamma,\ivcon{y}{l}{i}$. 
\end{itemize}

Given a functor $F \colon \cat{C} \To \cat{D}$, we say that the term interpretation $\alg{\Theta \Sep \Gamma \vdash t}{S}$ under the $\Sigma$-clone $\{ S \}_{n \in \N}$ in $\cat{C}$ is \emph{preserved} under $F$ if $F \alg{\Theta \Sep \Gamma \vdash t}{S} = \alg{\Theta \Sep \Gamma \vdash t}{FS}$ in $\cat{D}$. Evidently, term interpretations are preserved under cartesian functors.

For a second-order equational presentation $\E = (\Sigma, E)$, an \emph{$\E$-clone} in a cartesian category $\cat{C}$ is a $\Sigma$-clone $\{ S_n \}_{n \in \N}$ in $\cat{C}$ such that for all axioms $\Theta \Sep \Gamma \vdash_E s \eq t$ of $\E$, the morphisms $\alg{\Theta \Sep \Gamma \vdash s}{S}$ and $\alg{\Theta \Sep \Gamma \vdash t}{S}$ are equal in $\cat{C}$. In this case, we say that the clone $\{ S_n \}_{n \in \N}$ satisfies the axioms of $\E$. 
We write $\E \mbox{-} \Clone(\cat{C})$ for the full subcategory of $\Sigma \mbox{-}\Clone(\cat{C})$ consisting of the $\Sigma$-clones in $\cat{C}$ which satisfy the axioms of the presentation $\E = (\Sigma, E)$.

Clones for second-order signatures provide an axiomatisation for variable binding, parameterised metavariables and simultaneous substitution. We recall here that they are in fact an abstract, yet equivalent, formalisation of (set-theoretic) second-order model theory as presented in Section 3.6. 

\begin{prop}
For $\Sigma$ a second-order signature, the category $\Mod(\Sigma)$ of set-theoretic algebraic models for $\Sigma$ is equivalent to the category $\Sigma\mbox{-}\Clone(\Set)$ of $\Sigma$-clones in $\Set$.
\end{prop}

\begin{proof}
A detailed development of this equivalence appears in \cite{FiorePlotkinTuri1999}.
\end{proof}

One needs an additional argument to show that the same holds when adding equations, that is that clones and algebras for second-order equational presentations are equivalent. To this end, let $\E = (\Sigma,E)$ be a second-order equational presentation and $\mcon{m}{k} \Sep \vcon{x}{n} \vdash_\E s \eq t$ an equation of $\E$. Recall that a set-theoretic algebra $A$ of $\Mod(\E)$ satisfies all equations of $\E$, and therefore the respective term interpretations $\alg{s}{A}$ and $\alg{t}{A}$ are equal morphisms
$$ \prod_{i \in \card{k}} A^{\y m_i} \times \y n \To A $$
in $\Set^\ff$. Consequently, their corresponding exponential transposes $\lam \alg{s}{A}$ and $\lam \alg{t}{s}$ are equal morphisms
$ \prod_{i \in \card{k}} A^{\y m_i} \To A^{\y n} $.

Now, under the equivalence of Proposition 7.1, the $\Sigma$-algebra $A$ corresponds to the $\Sigma$-clone $\hat{A} = \{ A(n) \}_{n \in \N}$ in $\Set$, which induces the term interpretations $\alg{s}{\hat{A}}$ and $\alg{t}{\hat{A}}$ given by the component at $(0)$ of $\lam \alg{s}{A}$ and $\lam \alg{t}{A}$, respectively. Therefore, 
$$ \alg{s}{\hat{A}} = \alg{t}{\hat{A}} \colon \prod_{i \in \card{k}} A(m_i) \To A(n) $$
in $\Set$. We have thus shown that an equation of $\E = (\Sigma,E)$ satisfied by a $\Sigma$-algebra $A$ is also satisfied by the induced $\Sigma$-clone $\hat{A}$. 

The other direction is given by soundness and completeness. Suppose the judgement 
$$\mcon{m}{k} \Sep \vcon{x}{n} \vdash_\E s \eq t$$
is satisfied by a $\Sigma$-clone, then we know from soundness and completeness of Second-Order Equational Logic (Theorem 3.6) that it is necessarily satisfied by all $(\Sigma,E)$-algebras. 

A second-order term equation is hence satisfied by a signature algebra if and only if it is satisfied by the corresponding signature clone in $\Set$. This, together with Proposition 7.1, yields an alternative, yet equivalent, semantics of second-order equational presentations via abstract clone structures.

\begin{prop}
For $\E = (\Sigma,E)$ a second-order equational presentation, the categories $\Mod(\E)$ of second-order $\E$-algebras and $\E\mbox{-}\Clone(\Set)$ of set-theoretic $\E$-clones are equivalent.
\end{prop}

\subsection{Classifying Clones}

Before formalising second-order functorial semantics, we show that second-order algebraic theories, and in particular those that classify second-order equational presentations, come equipped with a canonical clone structure induced by their universal exponentiable object. This will enable us to link functorial models directly to (set-theoretic) algebraic models via these so-called classifying clone structures.

Let $\cat{C}$ be a cartesian category. An exponentiable object $C$ of $\cat{C}$ canonically induces the clone
$\langle C \rangle = \{ \e{C^n}{C} \}_{n \in \N}, \quad \langle C \rangle_n = (\e{C^n}{C})$
with variable maps $\iota_i^{(n)} \colon 1 \To \langle C \rangle_n$ given by the unique exponential mates of the cartesian projections

\vspace{-20pt}

\begin{diagram}[small]
1 \times C^n \cong C^n & \rTo^{\pi_i^{(n)}} & C \quad .
\end{diagram}
The substitution map $\varsigma_{m,n} \colon \langle C \rangle_m \times \langle C \rangle_n^m \To \langle C \rangle_n$ is given by the exponential mate of
\begin{diagram}[small]
(\e{C^m}{C}) \times (\e{C^n}{C^m}) \times C^n & \rTo^{(\e{C^m}{C})\times ev_{n,m}} & (\e{C^m}{C}) \times C^m & \rTo^{ev_m} & C \quad ,
\end{diagram}
where $ev_{n,m} \colon (\e{C^n}{C^m}) \times C^n \To C^m$ is the evaluation map associated with the exponential $\e{C^n}{C^m} = (\e{C^n}{C})^m$. 
We refer to $\langle C \rangle$ as the \emph{clone of elementary operations} on the object $C$ of $\cat{C}$. Thus, as it is the case with every clone, the family $\langle C \rangle$ canonically extends to a functor $\ff \To \cat{C}$ mapping $n$ to $ \langle C \rangle_n$ and $f \colon n \To m$ to $\e{C^f}{C} \colon \langle C \rangle_n \To \langle C \rangle_m$. \\

\textbf{Classifying clones.} Let $\Sigma$ be a second-order signature and $\M(\Sigma)$ its classifying category. The \emph{classifying clone} of a second-order signature $\Sigma$ is given by the clone of operations $\langle 0 \rangle = \{ (n) \}_{n \in \N}$ on the universal exponentiable object $(0)$ of $\M(\Sigma)$, together with the family
$\{ \tilde{f_\omega} \}_{\omega \colon \tuple{n}{l} \in \Sigma}$,
where for an operator $\omega \colon \tuple{n}{l}$, $f_\omega$ is the morphism
$$ \langle \omega \big( \dots, (x_1, \dots, x_{n_i}) \mvar{n}_i [x_1, \dots, x_{n_i}] ,\dots \big) \rangle \colon \tuple{n}{l} \To (0) $$
of $\M(\Sigma)$ and the instance at $j \in \N$ of the family 
$ \tilde{f_\omega} = \big\{ \big( \tilde{f_\omega} \big)_j \big\}_{j \in \N} $
is given by
\begin{diagram}[small]
(j+n_1, \dots, j+n_l) & \squad \cong \squad  & \e{(0)^j}{\tuple{n}{l}} & \squad \rTo^{\e{(0)^j}{f_\omega}} \squad & \e{(0)^j}{(0)} & \squad  \cong \squad & (j) \quad .
\end{diagram}

Clearly, classifying clones satisfy the properties of clone structures.

\begin{lem}
The canonical classifying clone of a second-order signature $\Sigma$ in its classifying category $\M(\Sigma)$ is a $\Sigma$-clone.
\end{lem}

\textbf{Classifying term interpretation.} The classifying clone $\langle 0 \rangle$ induces a canonical interpretation of terms in $\M(\Sigma)$. For $\Theta = (\mcon{m}{k})$ and $\Gamma = (\vcon{x}{n})$, the interpretation $\alg{t}{\langle 0 \rangle}$ of a term $\Theta \Sep \Gamma \vdash t$ under the classifying clone is the morphism
$ \langle t \rangle \colon \tuple{m}{k} \To (n) $
in $\M(\Sigma)$, which can be easily verified by induction on the structure of $t$. \\

\textbf{Classifying presentation clones.} For a second-order equational presentation $\E = (\Sigma, E)$, we define its classifying clone in its classifying category $\M(\E)$ in a similar fashion, namely by the clone of operations $\langle 0 \rangle$ together with the family $\{ (\tilde{f}_\omega)_n \}_{n \in \N}$, where for $\omega \colon n_1, \dots, n_l$, the morphism $f_\omega$ is taken to be the tuple of the equivalence of the same term as in the definition of classifying signature clones, more precisely
$$ \big\langle \big[ \omega \big( \dots, (x_1, \dots, x_{n_i}) \mvar{n}_i [x_1, \dots, x_{n_i}] ,\dots \big) \big]_\E \big\rangle \colon \tuple{n}{l} \To (0) \quad . $$

A similar inductive argument shows that the interpretation for a term $\Theta \Sep \Gamma \vdash t$ induced by the classifying clone $\langle 0 \rangle$ in $\M(\E)$ is the morphism $\langle [ \Theta \Sep \Gamma \vdash t ]_\E \rangle$. 

A derivable judgement $\Theta \Sep \Gamma \vdash_\E s \eq t$ of $\E$ is therefore satisfied by the classifying clone of $\E$ in $\M(\E)$, since $\langle [ \Theta \Sep \Gamma \vdash s ]_\E \rangle$ and $\langle [ \Theta \Sep \Gamma \vdash t ]_\E \rangle$ are equal morphisms in $\M(\E)$, and therefore $\alg{s}{\langle 0 \rangle} = \alg{t}{\langle 0 \rangle}$. Classifying clones hence provide \emph{sound semantics} for second-order equational presentations in their classifying categories.

\subsection{Second-Order Functorial Semantics}

We extend Lawvere's functorial semantics for algebraic theories \cite{Lawvere2004} from first to second order.

\begin{defn}[Second-Order Functorial Model]
A second-order \emph{functorial model} of a second-order algebraic theory $M \colon \M \to \Mlaw$ is given by a cartesian functor $\Mlaw \To \cat{C}$, for $\cat{C}$ a cartesian category. We write $\FMOD{M, \cat{C}}$ for the category of functorial models of $M$ in $\cat{C}$, with morphisms (necessarily monoidal) natural transformations between them. A second-order \emph{set-theoretic functorial model} of a second-order algebraic theory $M \colon \M \to \Mlaw$ is simply a cartesian functor from $\Mlaw$ to $\Set$. We write $\FMOD{M}$ for the category of set-theoretic functorial models of $M$ in $\Set$.
\end{defn}

Note that, just as in Lawvere's first-order definition, we mereley ask for preservation of the cartesian structure rather than strict preservation. Consequently, functorial models of the same second-order algebraic theory may differ only by the choice of the cartesian product in $\Set$. However, since the cartesian structure in $\Set$ is not strictly associative (whereas it is strictly associative in any first- and second-order algebraic theory), asking for preservation in the definition of a functorial model avoids the creation of unnatural categories of models.

\subsection{Second-Order Semantic Categorical Algebraic Theory Correspondence}

We show that classifying clones, and thus second-order algebras, correspond to second-order functorial models. 

\begin{prop}
Let $\E=(\Sigma,E)$ be a second-order equational presentation and $M_\E \colon \M \To \M(\E)$ its classifying algebraic theory, and let $\cat{C}$ be a cartesian category. The category of $\E$-clones $\E\mbox{-}\Clone (\cat{C})$ and the category of second-order functorial models $\FMOD{M_\E, \cat{C}}$ are equivalent.
\end{prop}

\begin{proof}
We provide an explicit description of the equivalence functors. Define
$$ \Upsilon \squad \colon \FMOD{M_\E, \cat{C}} \longrightarrow \E\mbox{-}\Clone(\cat{C})  $$
by mapping a cartesian functor $F \colon \M(\E) \To \cat{C}$ to the clone
$ \hat{F} \Def \{ F(n) \}_{n \in \N} $
whose structure maps are given by the image under $F$ of the structure maps of the canonical classifying clone $\langle n \rangle$ of $\M(\E)$. This makes $\hat{F}$ indeed a clone for the signature $\Sigma$, as, by Lemma ?, cartesian functors preserve clone structures. $\hat{F}$ is moreover a clone for the equational presentation $\E$, as it satisfies all equations in $\cat{C}$: given an equation $\Theta \Sep \Gamma \vdash s \eq t$ of $\E$, we have $F\langle[s]_\E\rangle = F\langle[t]_\E\rangle$ (since $\langle[s]_\E\rangle = \langle[t]_\E\rangle$), and therefore we get, by Lemma ?,  that $\alg{s}{\hat{F}} = F\alg{s}{\langle 0 \rangle} = F \alg{t}{\langle 0 \rangle} = \alg{t}{\hat{F}}$. 

On morphisms of $\FMOD{M_\E, \cat{C}}$, $\Upsilon$ is defined by mapping a monoidal natural transformation $\alpha \colon F \To G$ to $\{ \alpha_n \}_{n \in \N} \colon \{F(n)\}_{n \in \N} \To \{G(n)\}_{n \in \N}$. This is indeed a homomorphism of $\E$-clones because $\alpha$ is natural and the clone structure maps of $\hat{F}$ and $\hat{G}$ are the images of those of $\langle n \rangle$ under $F$ and $G$. Furthermore, note that $\Upsilon$ is functorial: the identity natural transformation $id^{(F)} \colon F \To F$ is mapped under $\Upsilon$ to $\{ id^{(F)}_n \}_{n \in \N}$, where each $id^{(F)}_n \colon F(n) \To F(n)$ is simply the identity morphism in $\cat{C}$. Similarly, for natural transformations $\alpha \colon F \To G$ and $\beta \colon G \To H$, the image of the composite $\beta \circ \alpha$ under $\Upsilon$ is $\{ (\beta \circ \alpha)_n \}_{n \in \N} = \{ \beta_n \circ \alpha_n \}_{n \in \N}$. 

In the other direction, define
$$ \bar{\Upsilon} \squad \colon \E\mbox{-}\Clone(\cat{C}) \longrightarrow \FMOD{M_\E, \cat{C}} $$
by mapping an $\E$-clone $\{ C_n \}_{n \in N}$ to the functor $F^{(C)} \colon \M(\E) \To \cat{C}$, which maps $\tuple{m}{k}$ to $C_{m_1} \times \cdots \times C_{m_k}$. For $\Theta = (\mcon{m}{k})$ and $\Gamma = (\vcon{x}{n})$, the image of the morphism 
$ \langle [\Theta \Sep \Gamma \vdash t ]_\E\rangle \colon \tuple{m}{k} \To (n) $
under $F^{(C)}$ is defined to be the interpretation $\alg{t}{C}$ of the term $t$ under the clone $C$. This definition respects the equivalence relation of $\E$ as given an equation $\Theta \Sep \Gamma \vdash_\E s \eq t$, we know that $\alg{s}{\langle n \rangle} = \alg{t}{\langle n \rangle}$ since $\langle n \rangle$ is an $\E$-clone, and therefore $F^{(C)}\langle[s]_\E\rangle = F^{(C)}\langle[t]_\E\rangle$ in $\cat{C}$. Note that $F^{(C)}$ is cartesian by definition. 

On morphisms of $\E\mbox{-}\Clone(\cat{C})$, $\bar{\Upsilon}$ is defined by mapping a clone homomorphism 
$\{ h_n \}_{n \in \N} \colon \{ C_n \}_{n \in \N} \To \{ D_n \}_{n \in \N}$
to $\bar{h} \colon F^{(C)} \To F^{(D)}$, with component at $\tuple{m}{k}$ given by $\bar{h}_{\tuple{m}{k}} = h_{m_1} \times \dots, \times h_{m_k}$. Because clone homomorphisms commute with the clone structure maps, we are ensured that $\bar{h}$ is a natural transformation. This can be seen more explicitly by induction on the term structure:
\begin{itemize}
\item[-] For $\langle [x_i]_\E\rangle \colon \tuple{m}{k} \To (n)$, the diagram
\begin{diagram}[small]
C_{m_1} \times \cdots \times C_{m_k} & \rTo^{\scriptstyle !} & 1 & \rTo^{\scriptstyle \iota_i^{(C)}} & C_n \\
\dTo^{\scriptstyle h_{m_1} \times \cdots \times h_{m_k}} & & \dTo^{\scriptstyle =} & & \dTo_{\scriptstyle h_n} \\
D_{m_1} \times \cdots \times D_{m_k} & \rTo^{\scriptstyle 1} & 1 & \rTo^{\scriptstyle \iota_i^{(D)}} & D_n
\end{diagram}
by uniqueness of the terminal map ! and because $h$ is a homomorphism of clones and hence commutes with the clone structure maps $\iota_i^{(-)}$.
\item[-] Similarly, for $\langle [\mvar{m}_i [t_1, \dots, t_{m_i}]]_\E \rangle \colon \tuple{m}{k} \To (n)$, the following diagram commutes
\begin{diagram}[small]
C_{m_1} \times \cdots \times C_{m_k} & \rTo^{\scriptstyle \big\langle \pi_i^{(C)}, F^{(C)}\langle[t_1]_\E\rangle, \dots, F^{(C)}\langle[t_{m_i}]_\E\rangle \big\rangle} & C_{m_i} \times (C_n)^{m_i} & \rTo^{\scriptstyle \varsigma^{(C)}_{m_i,n}} & C_n \\
\dTo^{\scriptstyle h_{m_1} \times \cdots \times h_{m_k}} & & \dTo^{\scriptstyle h_{m_i} \times (h_n)^{m_i}} & & \dTo_{\scriptstyle h_n} \\
D_{m_1} \times \cdots \times D_{m_k} & \rTo_{\scriptstyle \big\langle \pi_i^{(D)}, F^{(D)}\langle[t_1]_\E\rangle, \dots, F^{(D)}\langle[t_{m_i}]_\E\rangle \big\rangle} & D_{m_i} \times (D_n)^{m_i} & \rTo^{\scriptstyle \varsigma^{(D)}_{m_i,n}} & D_n 
\end{diagram}
by induction on $F^{(-)}\langle[t_j]_\E\rangle$ for all $j \in \card{m_i}$, by universality of the cartesian map $\pi_i^{(D)}$, and because $h_n$ commutes with the clone structure maps $\varsigma$.
\item[-] For $\omega \colon n_1, \dots, n_l$ and $\langle [ \omega\big( (\vvar{y}_1)t_1, \dots, (\vvar{y}_l)t_l \big) ]_\E \rangle$, the following diagram commutes for the same reasons as above:
\begin{diagram}[small]
C_{m_1} \times \cdots \times C_{m_k} & \rTo^{\scriptstyle \big\langle F^{(C)}\langle[t_1]_\E\rangle, \dots, F^{(C)}\langle[t_l]_\E\rangle \big\rangle} & C_{n+n_1} \times \cdots \times C_{n+n_l} & \rTo^{\scriptstyle \tilde{\omega}^{(C)}} & C_n \\
\dTo^{\scriptstyle h_{m_1} \times \cdots \times h_{m_k}} & & \dTo^{\scriptstyle h_{n+n_1} \times \cdots \times h_{n+n_l} } & & \dTo_{\scriptstyle h_n} \\
D_{m_1} \times \cdots \times D_{m_k} & \rTo_{\scriptstyle \big\langle F^{(D)}\langle[t_1]_\E\rangle, \dots, F^{(D)}\langle[t_l]_\E\rangle \big\rangle} & D_{n+n_1} \times \cdots \times D_{n+n_l} & \rTo^{\scriptstyle \tilde{\omega}^{(D)}} & D_n \\
\end{diagram}
\end{itemize}
That $\bar{\Upsilon}$ is functorial follows from the fact that natural transformations in $\FMOD{M_\E,\cat{C}}$ are monoidal. More precisely, an identity homomorphism of clones $\{ id_n \}_{n \in \N}$ is mapped under $\bar{\Upsilon}$ to the identity natural transformation with component at $\tuple{m}{k}$ given by $id_{m_1} \times \cdots \times id_{m_k}$, which is equal to $id_{\tuple{m}{k}}$. Similarly, a composite of clone homomorphisms $\{ (g \circ h)_n \}_{n \in \N}$ is mapped to $(\overline{g \circ h})$ with component at $\tuple{m}{k}$ given by
$$ (g \circ h)_{m_1} \times \cdots \times (g \circ h)_{m_k} \squad = \squad (g \circ h)_{\tuple{m}{k}} \squad = \squad g_{\tuple{m}{k}} \circ h_{\tuple{m}{k}} \quad . \\$$
The functors $\Upsilon$ and $\bar{\Upsilon}$ are indeed equivalences. A functorial model $F \colon \M(\E) \To \cat{C}$ is mapped under $\bar{\Upsilon} \circ \Upsilon$ to $F^{(\hat{F})} \colon \M(\E) \To \cat{C}$, which maps an object $\tuple{m}{k}$ to $F(m_1) \times \cdots \times F(m_k) \cong F\tuple{m}{k}$ and a morphism $\langle [\Theta \Sep \Gamma \vdash t]_\E \rangle$ to $\alg{t}{\hat{F}} = F \alg{t}{\langle 0 \rangle} = F \langle [\Theta \Sep \Gamma \vdash t]_\E \rangle$. A natural transformation $\alpha \colon F \To G$ is mapped under $\bar{\Upsilon} \circ \Upsilon$ to $\hat{\alpha} \colon F^{(\hat{F})} \To F^{(\hat{G})}$ and, because it is monoidal, has component at $\tuple{m}{k}$ given by $\hat{\alpha}_{\tuple{m}{k}} = \alpha_{m_1} \times \cdots \times \alpha_{m_k} = \alpha_{\tuple{m}{k}}$.
In the other direction, an $\E$-clone $\{C_n\}_{n \in \N}$ is mapped under $\Upsilon \circ \bar{\Upsilon}$ to the clone $\hat{F}^{(C)} = \{ F^{(C)}(n) \}_{n \in \N} = \{ C_n \}_{n \in \N}$, and an $\E$-clone homomorphism $\{ h_n \}_{n \in \N} \colon \{ C_n \}_{n \in \N} \To \{ D_n \}_{n \in \N}$ to $\{ \bar{h}_{(n)} \}_{n \in \N} = \{ h_n \}_{n \in \N}$.
\end{proof}

If we take the cartesian category $\cat{C}$ to be $\Set$, we immediately obtain the correspondence between set-theoretic functorial models, models for equational presentations, and set-theoretic clone structures.

\begin{thm}[Second-Order Semantic Categorical Algebraic Theory Correspondence]
For every second-order equational presentation $\E$, the category $\Mod(\E)$ of $\E$-models and the category of second-order functorial models $\FMOD{M_\E}$ are equivalent.
\end{thm}

Using the Second-Order Syntactic Categorical Algebraic Theory Correspondence, we have the following equivalent formulation of the above semantic correspondence.

\begin{cor}
For every second-order algebraic theory $M \colon \M \To \Mlaw$, the category of second-order functorial models $\FMOD{M}$ and the category of algebraic models $\Mod(\IL(M))$ are equivalent.
\end{cor}

\subsection{Translational Semantics}

Second-order functorial semantics enables us to take a model of an algebraic theory in any cartesian category $\cat{C}$. This way of abstractly defining algebras for theories has a syntactic counterpart via syntactic translations, which we refer to as second-order \emph{translational semantics}. 

Consider two second-order equational presentations $\E$ and $\E '$, their corresponding classifying algebraic theories $M_\E \colon \M \To \M(\E)$ and $M_{\E '} \colon \M \To \M(\E ')$, and let $\trans \colon \E \To \E '$ be a second-order syntactic translation. Note that its induced algebraic translation $\M(\trans) \colon \M(\E) \To \M(\E ')$, which commutes with the theories $M_\E$ and $M_{\E '}$, is by definition a second-order functorial model of the theory $M_\E$ in the cartesian category $\M(\E ')$. The canonical notion of a morphism of (second-order) algebraic theories is thereby intuitively providing a model of one algebraic theory into another. 

From the categorical equivalence of the Syntactic Categorical Algebraic Theory Correspondence, second-order syntactic translations can be thought of as \emph{syntactic} notions of models of one equational presentation into another. Therefore, by explicitly defining the translation $\trans \colon \E \To \E '$, we implicitly provide a model of the presentation $\E$ in $\E '$. 

We have in this work reviewed first- and second-order set-theoretic semantics for equational presentations, as well as categorical semantics, and finally introduced second-order functorial semantics. Through the development of syntactic translations, we have thus introduced a less abstract, more concrete way of giving semantics to equational presentations. We refer to this as (second-order) \emph{Translational Semantics}.

\section{Concluding Remarks}
\label{conclusion}

We have incorporated second-order languages into universal algebra by developing a programme from the viewpoint of Lawvere's algebraic theories. 
The pinnacle of our work is the notion of \emph{second-order algebraic theory}, which we defined on top of a base category, the second-order theory of equality $\M$, representing the elementary operators and equations present in every second-order language. We showed that $\M$ can be described abstractly via the universal property of being the free cartesian category on an exponentiable object. 
At the syntactic level, we established the correctness of our definition by showing a categorical equivalence between second-order equational presentations and second-order algebraic theories. This equivalence, referred to as the Second-Order Syntactic Categorical Algebraic Theory Correspondence, involved distilling a notion of syntactic translation between second-order equational presentations that corresponds to the canonical notion of morphism between second-order algebraic theories. Syntactic translations provide a mathematical formalisation of notions such as encodings and transforms for second-order languages.  
On top of this syntactic correspondence, we furthermore established the Second-Order Semantic Categorical Algebraic Theory Correspondence. This involved generalising Lawvere's notion of functorial model of algebraic theories to the second-order setting. By this semantic correspondence, second-order functorial semantics correspond to the model theory of second-order universal algebra. 

With this theory in place, one is now in a position to: (i) consider constructions on second-order equational presentations in a categorical setting, and indeed the developments for Lawvere theories on lmits, colimits, and tensor product carry over to the second-order setting; (ii) investigate conservative-extension results for second-order equational presentations in a mathematical framework; and (iii) study Morita equivalence for second-order algebraic theories.

We conclude by outlining how the significant theory of \emph{algebraic functors} surrounding first-order algebraic theories extends to the second-order universe.

\subsection{Second-Order Algebraic Functors}

The concept of an algebraic functor arising from morphisms of Lawvere theories has been developed by Lawvere \cite{Lawvere2004} and revisited many times since then \cite{Borceux1994, AdamekRosickyVitale2009}. It is the canonical notion of morphism between algebraic categories.

\begin{defn}[Algebraic Categories and Functors]
A category is called \emph{algebraic} if it is equivalent to the category of functorial models $\FMod(L)$ for some algebraic theory $L \colon \F \To \Law$.
Let $F \colon L \To L'$ be an algebraic translation of (mono-sorted first-order) Lawvere theories $L \colon \F \To \Law$ and $L' \colon \F \To \Law '$. The functor
$ \FMod(F) \colon \FMod(L') \To \FMod(L) \colon G \mapsto G \circ F$
is called an \emph{algebraic functor}.
\end{defn}

We obtain the following commutative diagram, where the unlabelled arrows are the canonical (monadic) forgetful functors:
\begin{diagram}[small]
\FMod(L') & & \rTo^{\FMod(F)} & & \FMod(L)\\
& \rdTo^{\quad} &  & \ldTo^{\quad} & \\
& & \Set & &
\end{diagram}

A functor of algebraic categories is algebraic if and only if it is induced by a morphism of algebraic theories, making the two definitions equivalent. Moreover, it can be shown that algebraic functors have left adjoints. This is an immediate consequence of the Adjoint Lifting Theorem. 
The resulting \emph{algebraic adjunction} is in fact monadic, and we refer the reader to \cite{Maclane1998} and \cite{Borceux1994} for more details.

\begin{prop}
Let $F \colon L_1 \To L_2$ be an algebraic translation of algebraic theories $L_1 \colon \F \To \Law_1$ and $L_2 \colon \F \To \Law_2$. Then its induced algebraic functor $\FMod(F) \colon \FMod(L_2) \To \FMod(L_1)$ has a left adjoint $\widetilde{F} \colon \FMod(L_1) \To \FMod(L_2)$. 
\end{prop}

This left adjoint $\widetilde{F}$ is the essentially unique functor which preserves sifted colimits and makes the following diagram commute up to natural isomorphism.
\begin{diagram}[small]
\Law_1^\mathrm{op} & \rTo^{F^\mathrm{op}} & \Law_2^\mathrm{op} \\
\dTo^{Y_{L_1}} & & \dTo_{Y_{L_2}} \\
\FMod(L_1) & \rTo^{\widetilde{F}} & \FMod(L_2) 
\end{diagram}

The algebraic importance of these left adjoints is pointed out by Lawvere in his thesis \cite{Lawvere2004}. For example, the adjoint to the algebraic functor induced by an algebraic translation from the theory of monoids to the theory of rings essentially assigns to a monoid $M$ the monoid ring $Z[M]$ with integer coefficients. The fact that these adjoints form the commutative diagram above implies, for instance, that a free ring can be constructed either as the monoid ring of a free monoid, or as the tensor ring of a free abelian group. These are well-known facts from universal algebra, but given a more abstract formulation via algebraic functors and their adjoints.

Just as in the first-order case, every algebraic translation $F \colon M \To M'$ between second-order algebraic theories $M \colon \M \To \Mlaw$ and $M' \colon \M \To \Mlaw '$ contravariantly induces a \emph{second-order algebraic functor} $\FMOD{F} \colon \FMOD{M'}\To\FMOD{M} \squad ; \squad S \mapsto S \circ F$ between the corresponding categories of second-order functorial models. 

\begin{thm}
The algebraic functor $\FMOD{F} \colon \FMOD{M'} \To \FMOD{M}$ induced by a second-order algebraic translation $F \colon M \To M'$ has a left adjoint, and the resulting adjunction is monadic.
\end{thm}

Syntactic translations of second-order equational presentations similarly yield a notion of algebraic functor which is naturally isomorphic to the one introduced above. Observe that second-order syntactic signature translations behave essentially as natural transformations between the corresponding signature endofunctors and their induced monads: for second-order signatures $\Sigma_1$ and $\Sigma_2$, let $\sfun{\Sigma_1}$ be the signature endofunctor induced by $\Sigma_1$, and $\mon_{\Sigma_2}$ the (underlying functor of the) induced monad corresponding to $\Sigma_2$. More precisely, in the situation
$$ \xymatrix{\ar @{} [r] |{\displaystyle \perp} \Set^\ff \ar@(dl,dr)[]_{\displaystyle \sfun{\Sigma_2}} \ar@/^1pc/[r]^{}  & \Mod(\Sigma_2) \ar@/^1pc/[l]^{}  } $$
$\mon_{\Sigma_2}$ is the monad induced by the above adjunction, so that $\Alg{\mon_{\Sigma_2}} \squad \cong \squad \Mod(\Sigma_2)$. 

A translation $\trans \colon \Sigma_1 \rightarrow \Sigma_2$ induces a natural transformation $\alpha^\trans \colon \sfun{\Sigma_1} \rightarrow \mon_{\Sigma_2}$, which is \emph{strong} in the sense that

\vspace{-20pt}

\begin{diagram}[small]
\sfun{\Sigma_1}(X) \tensor Y & \rTo^{\displaystyle s_{\sfun{\Sigma_1}}} & \sfun{\Sigma_1}(X \tensor Y)\\
\dTo^{\displaystyle \alpha^\trans_X \tensor Y} & & \dTo_{\displaystyle \alpha^\trans_{X \tensor Y}} \\
\mon_{\Sigma_2}(X) \tensor Y & \rTo^{\displaystyle s_{\mon_{\Sigma_2}}} & \mon_{\Sigma_2}(X \tensor Y)
\end{diagram}
commutes for the canonical pointed strengths $s_{\sfun{\Sigma_1}}$ and $s_{\mon_{\Sigma_2}}$. 
Natural transformations induced in this way by syntactic translations contravariantly induce algebraic functors between categories of set-theoretic algebras. For $\trans \colon \Sigma_1 \To \Sigma_2$ a second-order translation with induced natural transformation $\alpha^\trans \colon \sfun{\Sigma_1} \To \mon_{\Sigma_2}$, let $A \in \Mod(\Sigma_2)$ be a $\Sigma_2$-model, with monoid structure $\nu_A \colon \y 1 \To A$ and $\varsigma_A \colon A \tensor A \To A$, and $\sfun{\Sigma_2}$-algebra structure map given by $\varphi_A \colon \sfun{\Sigma_2}A \To A $. Denote by $\squad \delta_A \colon \mon_{\Sigma_2}A \To A \squad$ the corresponding $\mon_{\Sigma_2}$-algebra structure map induced by the categorical equivalence $\squad \Mod(\Sigma_2) \squad \cong \squad \Alg{\mon_{\Sigma_2}}$. 
Composing this $\mon_{\Sigma_2}$-algebra structure map $\delta$ with natural transformations $\sfun{\Sigma_1} \To \mon_{\Sigma_2}$ essentially defines the mapping of algebraic functors. More precisely, a second-order signature translation $\trans \colon \Sigma_1 \To \Sigma_2$ yields the algebraic functor
$ \Mod(\tau) \colon \Mod(\Sigma_2) \To \Mod(\Sigma_1) $
by mapping $A \in \Set^\ff$ with structure maps $\nu_A \colon \y 1 \To A$, $\varsigma_A \colon A \tensor A \To A$, and $\varphi_A \colon \sfun{\Sigma_2}A \To A $ to the algebra with same underlying presheaf $A$ and same monoid maps $\nu_A$ and $\varsigma_A$, but with $\sfun{\Sigma_1}$-algebra structure map given by the composite $\delta_A \circ \alpha^\trans_A$.
This morphism is compatible with the monoid structure given by $\nu_A$ and $\varsigma_A$ because of the strength of the natural transformation $\alpha^\trans$ discussed above.

The compatibility of the monoid structure with the structure map of the signature endofunctor can be viewed as an abstract description of the compositionality of syntactic translations with substitution and metasubstitution. The algebraic functor $\Mod(\tau)$ clearly commutes with the canonical forgetful functors into $\Set^\ff$. Using a similar argument as in the first-order universe, we obtain a left adjoint to $\Mod(\tau)$, with the resulting adjunction being monadic.

Next, we use the notion of algebraic equational systems developed by Fiore and Hur in \cite{FioreHur2007, FioreHur2008a} to derive algebraic functors induced by syntactic translations of second-order equational presentations.

\begin{defn}[Equational System]
An \emph{equational system} $\mathbb{S}$ is given by a pair of functors $L,R \colon \Alg{F} \To \Alg{D}$ between categories of algebras for endofunctors over some base category $\cat{C}$. In the framework of equational presentations, the \emph{functorial signature} $F$ is a generalisation of the concept of endofunctor induced by an algebraic signature; the so-called \emph{functorial terms} $L,R$ generalise the notion of equation; and the endofunctor $D$ corresponds to the arity of the equation. The category $\Alg{\mathbb{S}}$ of algebras for the equational system $\mathbb{S}$ is given by the equaliser $\Alg{\mathbb{S}} \hookrightarrow \Alg{F}$ of $L,R$. More explicitly, an $\mathbb{S}$-algebra is simply an $F$-algebra $(A, a\colon FA \To A)$ such that $L(A,a)$ and $R(A,a)$ are equal $D$-algebras on $A$.
\end{defn}

For a second-order signature $\Sigma$, the equational systems formalism allows one to write

\vspace{-30pt}

\begin{diagram}[small]
\Mod(\Sigma) & \rInto^{\mathrm{eq}} & \Alg{\Sfun{\Sigma}} & \pile{\rTo^{} \\ \rTo^{}} & \Alg{\Gamma_\Sigma} \quad ,
\end{diagram}
where $\Sfun{\Sigma}(X) = \sfun{\Sigma}(X) + V + X \tensor X$, and the parallel pair encodes the equations of $\Sigma$-monoids. For a second-order equational presentation $\E = (\Sigma,E)$, we further have 

\vspace{-25pt}

\begin{diagram}[small]
& & \Mod(\E) & & \\
& & \dInto^{\mathrm{eq}} & & \\
\Alg{\Gamma_\Sigma} & \pile{\lTo{} \\ \lTo^{}} & \Alg{\Sfun{\Sigma}} & \pile{\rTo^{}\\ \rTo^{}} & \Alg{\Gamma_E} \quad ,
\end{diagram}
where the left parallel pair encodes the $\Sigma$-monoids (or substitution structure) as above, and the parallel pair to the right encodes the equations in $E$. We therefore get the equivalent equaliser diagram

\vspace{-25pt}

\begin{diagram}[small]
\Mod(\E) & \rInto^{\mathrm{eq}} & \Alg{\Sfun{\Sigma}} & \pile{\rTo^{} \\ \rTo^{}} & \Alg{(\Gamma_\Sigma + \Gamma_E)} \quad ,
\end{diagram}
so that in fact one has

\vspace{-25pt}

\begin{diagram}[small]
\Mod(\E) & \rInto^{\mathrm{eq}} & \Mod(\Sigma) & \pile{\rTo^{} \\ \rTo^{}} & \Alg{\Gamma_\E} \quad .
\end{diagram}

The category $\Alg{\mathbb{S}_\E}$ of algebras for $\mathbb{S}_\E$ is isomorphic to the category $\Mod(\E)$ of models for the equational presentation $\E$. Moreover, $\Alg{\mathbb{S}_\E}$ is a cocomplete, full reflective subcategory of $\Alg{\sfun{\Sigma}}$. The forgetful functor $\Alg{\mathbb{S}_\E} \To \Set^\ff$ has a left adjoint, and the resulting adjunction is monadic. We refer the reader to \cite{FioreHur2007, FioreHur2008a} for more details.

We use this framework to derive algebraic functors between categories of models for second-order equational presentations, or equivalently, for equational systems. To this end, let $\E_1 = (\Sigma_1,E_1)$ and $\E_2 = (\Sigma_2, E_2)$ be second-order equational presentations, and $\trans \colon \E_1 \To \E_2$ a syntactic translation. Consider the following diagram:

\vspace{-25pt}

\begin{diagram}[small]
& & \Mod(\E_2) & \rInto^{J_2} && \Mod(\Sigma_2) & \pile{\rTo^{L_2} \\ \rTo_{R_2}} && \Alg{\Gamma_{\E_2}} \\ & \ldTo \\ \Set^\ff & & \dDashto_{\displaystyle\Mod(\trans)}& & & \dTo_{\displaystyle \Mod(\trans ')} \\ & \luTo \\ 
& & \Mod(\E_1) & \rInto^{J_1} && \Mod(\Sigma_1) & \pile{\rTo^{L_1} \\ \rTo_{R_1}} && \Alg{\Gamma_{\E_1}}
\end{diagram}

Here, $\trans ' \colon \Sigma_1 \To \Sigma_2$ is the restriction of $\trans$ to the underlying signatures of $\E_1$ and $\E_2$, and  $\Mod(\trans ')$ is the induced algebraic functor $\Mod(\Sigma_2) \To \Mod(\Sigma_1)$, as derived above. 
$\Mod(\E_2)$ together with the composite functor $\Mod(\trans ') \circ J_2$ equalise the pair $L_1,R_1$, roughly because axioms of $\E_1$ are mapped via the syntactic translation $\trans$ to theorems of $\E_2$. Hence, one gets the unique functor $\Mod(\trans)$ making the above diagram commute. Furthermore, by the Adjoint Lifting Theorem and the monadicity result of Theorem 8.3, this functor will have a left adjoint, and the resulting adjunction is monadic. 

We refer to $\Mod(\trans) \colon \Mod(\E_2) \To \Mod(\E_1)$ as the second-order \emph{syntactic algebraic functor} induced by the syntactic translation $\trans \colon \E_1 \To \E_2$. Using the Second-Order Semantic Categorical Algebraic Theory Correspondence, this functor is naturally isomorphic to the composite

\vspace{-20pt}

\begin{diagram}[small]
\Mod(\E_2) \squad \cong \squad \FMOD{M_{\E_2}} & \rTo^{\FMOD{\M(\trans)}} &  \FMOD{M_{\E_1}} \squad \cong \squad \Mod(\E_1) \quad ,
\end{diagram} 
where for $i=1,2$, $M_{\E_i} \colon \M \To \M(\E_i)$ is the algebraic theory classifiying $\E_i$, $\M(\trans)$ is the algebraic translation induced by $\trans$, and $\FMOD{\M(\trans)}$ is its induced second-order algebraic functor.

\appendix

\section{Second-Order Substitution and Metasubstitution Lemmas}

\begin{lem}[Second-Order Substitution Lemma]
Given terms
$$ \Theta \sep \Gamma \vdash s_i \quad (1 \leq i \leq n), \qquad \Theta \sep \Gamma \vdash r_j \quad (1 \leq j \leq k), \qquad \mathit{and} \qquad \Theta \sep \vcon{x}{n}, \vcon{y}{k} \vdash t , $$
we have
$$ \Theta \sep \Gamma \vdash t \big\{ x_i \Def s_i \big\}_{i \in \card{n}} \big\{ y_j \Def r_j \big\}_{j \in \card{k}} \squad = \squad t \Big\{ x_i \Def s_i \big\{ y_j \Def r_j \big\}_{j \in \card{k}} \Big\}_{i \in \card{n}} \quad . $$
\end{lem}

\begin{lem}[Substitution-Metasubstitution Lemma]
Given terms
$$ \mcon{m}{k} \sep \Gamma \vdash t_i \quad (1 \leq i \leq n), \qquad \Theta \sep \Gamma, \vvar{y}_j \vdash s_j \quad (1 \leq j \leq k),$$ 
$$ \mathit{and} \qquad \hspace{75pt} \mcon{m}{k}\sep \vcon{x}{n} \vdash t , \qquad \hspace{87pt} $$
we have
\begin{eqnarray*}
\Theta \sep \Gamma & \vdash & t \big\{ x_i \Def t_i \big\}_{i \in \card{n}} \big\{ \mvar{m}_j \Def (\vvar{y}_j) s_j \big\}_{j \in \card{k}} \\
& = & t \big\{ \mvar{m}_j \Def (\vvar{y}_j) s_j \big\}_{j \in \card{k}} \Big\{ x_i \Def t_i \big\{ \mvar{m}_j \Def (\vvar{y}_j) s_j \big\}_{j \in \card{k}} \Big\}_{i \in \card{n}} \quad . 
\end{eqnarray*}
\end{lem}

\begin{lem}[Metasubstitution Lemma I]
Given terms
$$ \Theta \sep \Gamma, \vvar{x}_i \vdash r_i \quad (1 \leq i \leq k), \qquad \Theta \sep \Gamma, \vvar{y}_j \vdash s_j \quad (1 \leq j \leq l),$$ 
$$ \mathit{and} \qquad \hspace{75pt} \mcon{m}{k}, \mcon{n}{l} \sep \Gamma \vdash t , \qquad \hspace{87pt} $$
we have 
\begin{eqnarray*}
\Theta \sep \Gamma & \vdash & t \big\{ \mvar{m}_i \Def (\vvar{x}_i) r_i \big\}_{i \in \card{k}} \big\{ \mvar{n}_j \Def (\vvar{y}_j) s_j \big\}_{j \in \card{l}} \\
& = & t \big\{ \mvar{n}_j \Def (\vvar{y}_j) s_j \big\}_{j \in \card{l}} \Big\{ \mvar{m}_i \Def (\vvar{x}_i) r_i \big\{ \mvar{n}_j \Def (\vvar{y}_j) s_j \big\}_{j \in \card{l}} \Big\}_{i \in \card{k}} \quad . 
\end{eqnarray*}
\end{lem}

\begin{lem}[Metasubstitution Lemma II]
Given terms
$$ \mcon{m}{k} \sep \Gamma \vdash t \qquad \mathit{and} \qquad \mcon{m}{k} \sep \Gamma, \ivcon{x}{m}{i} \vdash \mvar{m}_i [\ivcon{x}{m}{i}] $$
for $1 \leq i \leq k$, we have
$$ \mcon{m}{k} \sep \Gamma \vdash t \big\{ \mvar{m}_i \Def (\vvar{x}_i) \mvar{m}_i [\ivcon{x}{m}{i}] \big\}_{i \in \card{k}} = t \quad . $$
\end{lem}

\section{Proof of the Second-Order Theory/Presentation Correspondence}

We prove the correspondence via an explicit description of the isomorphism and its inverse. Define the identity-on-objects functor $ \nat_M \colon \Mlaw \To \M(\IL(M)) $
by mapping $f \colon \tuple{m}{k} \To (n)$ of $\Mlaw$ to 
$ \big\langle  \big[ \mcon{m}{k} \Sep \vcon{x}{n} \vdash \big[\term_f \big]_{\IL(M)} \big\rangle \colon \tuple{m}{k} \To (n) $.
Functoriality of $\nat_M$ is implied by the equational theory of $\IL(M)$. More precisely, the identity $id^\Mlaw_{\tuple{m}{k}}$ on $\tuple{m}{k}$ in $\Mlaw$ is mapped to the $k$-tuple of equivalence classes of
\begin{eqnarray*}
\mcon{m}{k} \Sep \ivcon{x}{m}{i} & \vdash & \term_{\pi^{(\Mlaw)}_i}  \\
& = & \term_{M(\pi^{(\M)}_i)}  \\
& = & \term_{M \langle \mvar{m}_i [\ivcon{x}{m}{i}] \rangle} \\
& \stackrel{\E 1}{\eq} & \mvar{m}_i [\ivcon{x}{m}{i}] \quad ,
\end{eqnarray*}
for $1 \leq i \leq k $ and $\pi^{(-)}_i \colon \tuple{m}{k} \To (m_i)$ the canonical projection in $-$, which makes the above tuple indeed the identity in $\M(\IL(M))$. Similarly, preservation of composition is a consequence of $(\E 2)$ of $\IL(M)$. Consider, without loss of generality, the morphisms $\langle f_1, \dots, f_l \rangle \colon \tuple{m}{k} \To \tuple{n}{l}$ and $g \colon \tuple{n}{l} \To (n)$ of $\Mlaw$. Then $\nat_M (g) \circ \nat_M (\langle f_1, \dots, f_l \rangle)$ is given by the equivalence class of 
\begin{eqnarray*}
\mcon{m}{k} \Sep \vcon{x}{n} & \vdash & \term_g \big\{ \mvar{n}_i \Def (\vvar{y}_i) \term_{f_i} \big\}_{i \in \card{l}}\\
& \stackrel{\E 2}{\eq} & \term_{g \circ \langle f_1, \dots, n_l \rangle} \quad ,
\end{eqnarray*}
making $\nat_M (g) \circ \nat_M (\langle f_1, \dots, f_l \rangle) = \nat_M (g \circ \langle f_1, \dots, n_l \rangle)$. 
This definition is strong enough to yield an algebraic translation from $M \colon \M \To \Mlaw$ to the classifying algebraic theory $M_{\IL(M)} \colon \M \To \M(\IL(M))$, since for any $\langle t \rangle \colon \tuple{m}{k} \To (n)$ in $\M$, the morphism $M\langle t \rangle \colon \tuple{m}{k} \To (n)$ in $\Mlaw$ is mapped under $\nat_M$ to the equivalence class of 
$ \mcon{m}{k} \Sep \vcon{x}{n} \vdash \term_{M\langle t \rangle}$,
which by $(\E 1)$ is provably equal to $t$, whose equivalence class is the image of $t$ under $M_{\IL(M)}$. 

In the other direction, define the identity-on-objects mapping
$ \natbar_M \colon \M(\IL(M)) \To \Mlaw $
by induction on the structure of representatives of equivalence classes $[-]_{\IL(M)}$ as follows:
\begin{itemize}
\item[-] $\big[\mcon{m}{k} \Sep \vcon{x}{n} \vdash x_i\big]_{\IL(M)}$ is mapped to 
\begin{diagram}[small]
\tuple{m}{k} & \rTo^{!^{(\Mlaw)}} & () & \rTo^{\lam(\pi_i^{(\Mlaw)} \circ \cong)} & (n) \quad .
\end{diagram}
\item[-] $\big[\mcon{m}{k} \Sep \vcon{x}{n} \vdash \mvar{m}_i [t_1, \dots, t_{m_i}]\big]_{\IL(M)}$ is mapped to 
\begin{diagram}[small]
\tuple{m}{k} & \rTo^{\big\langle \pi_i^{(\Mlaw)}, \natbar_M ([t_1]_{\IL(M)}), \dots, \natbar_M ([t_{m_i}]_{\IL(M)}) \big\rangle} & (m_i, n^{m_i}) &\rTo^{\varsigma^{(\Mlaw)}_{m_i,n}} & (n) \quad .
\end{diagram} 
\item[-] For $f \colon \tuple{n}{l} \To (j)$ in $\Mlaw$, 
$$ \big[\mcon{m}{k} \Sep \vcon{x}{n} \vdash \omega_f \big( (\vvar{y}_1) t_1, \dots, (\vvar{y}_l) t_l, s_1, \dots, s_j \big) \big]_{\IL(M)} $$
is mapped under $\natbar_M$ to the composite
\begin{diagram}[small]
\tuple{m}{k} & \rTo^{ \big\langle \natbar_M[t_1]_{\IL(M)}, \dots, \natbar_M[t_l]_{\IL(M)}, \natbar_M[s_1]_{\IL(M)}, \dots, \natbar_M[s_j]_{\IL(M)} \big\rangle} & (n+n_1, \dots, n+n_l, n^j) \\ 
& & \dTo^{\e{(0)^n}{\big(\ev_j \circ (f \times (0)^j)\big)}}\\
& & (n) 
\end{diagram}
\end{itemize}
Equivalence classes of elementary terms $s$ are simply mapped to $M\langle s \rangle$ under $\natbar_M$. 
We show that the mapping $\natbar_M$ is: (i) well-defined, (ii) functorial, and (iii) an algebraic translation $\M(\IL(M)) \To \Mlaw$. 

(i) To verify that $\natbar_M$ is well-defined, we show that equal terms (that is representatives of equivalence classes $[-]_{\IL(M)}$) according to axioms $(\E 1)$ and $(\E 2)$ of $\IL(M)$ are mapped under $\natbar_M$ to equal morphisms of $\Mlaw$. Consider axiom $(\E 1)$, and let $\langle s \rangle \colon \tuple{m}{k} \To (n)$ be a morphism of $\M$. Then the image of $\big[ \term_{M\langle s \rangle} \big]_{\IL(M)}$ under $\natbar_M$ is the composite 

\vspace{-20pt}

\begin{diagram}[small]
\tuple{m}{k} & \rTo^{\lam(id_{(m_1, \dots, m_k, 0^n)})} & \e{(0)^n}{(m_1, \dots, m_k, 0^n)} & \rTo^{\e{(0)^n}{\big( \ev_n \circ (M\langle s \rangle \times (0)^n) \big)}} & (n) \quad ,
\end{diagram}
which is simply $M\langle s \rangle$, and is in turn the image of $\langle s \rangle$ under $\natbar_M$ as $s$ is an elementary term. For the axiom $(\E 2)$, let $g \colon \tuple{n}{l} \To (n)$, $h \colon \tuple{m}{k} \To (n)$, and $f_i \colon \tuple{m}{k} \To (n_i)$ (for $1 \leq i \leq l$) be morphisms of $\Mlaw$ such that $g \circ \langle f_1, \dots, f_l \rangle = h$. Then 
\begin{eqnarray*}
& \quad & \natbar_M \Big( \big[ \term_g \big\{ \mvar{m}_i \Def (\vvar{x_i}) \term_{f_i} \big\}_{i \in \card{l}} \big]_{\IL(M)} \Big) \\ 
& = & \big( \e{(0)^n}{\big(\ev_n \circ \big( g \times (0)^n \big)\big)} \big) \circ \big( \e{(0)^n}{\big(\ev_n \circ \big(\langle f_1, \dots, f_l \rangle \times (0)^n \big)\big)} \big) \circ \lam (id_{(m_1, \dots, m_k, 0^n)}) \\
& = &  \big( \e{(0)^n}{\big(\ev_n \circ \big( (g \circ\langle f_1, \dots, f_l \rangle) \times (0)^n \big)\big)} \big) \circ \lam (id_{(m_1, \dots, m_k, 0^n)}) \\
& = & \big( \e{(0)^n}{\big(\ev_n \circ \big(h \times (0)^n \big)\big)} \big) \circ \lam (id_{(m_1, \dots, m_k, 0^n)}) \\
& = & \natbar_M \Big( \big[ \term_h  \big]_{\IL(M)} \Big) \quad .
\end{eqnarray*}

(ii) For the identity condition of functoriality, note that the identity in $\M(\IL(M))$ is given by the equivalence class of an elementary term, and by definition, a morphism $f = \langle [t]_{\IL(M)} \rangle$ of $\M(\IL(M))$, for $t$ an elementary term, is simply mapped to $M(\langle t \rangle)$ under $\natbar_M$. Therefore, for any $\tuple{m}{k}$ in $\M(\IL(M))$, and since $M$ is a functor, we have that
$\natbar_M\big(id^{\M(\IL(M))}_{\tuple{m}{k}}\big) = M\big(id^\M_{\tuple{m}{k}}\big) = id^\Mlaw_{\tuple{m}{k}} $,
where the superscript in $id^{\cat{C}}$ identifies the category $\cat{C}$ the identity is being taken in. For compositionality, note that, by its definition, $\natbar_M$ commutes with metasubstitution. More precisely, from the equational theory of $\IL(M)$, any morphism of $\M(\IL(M))$ can be written as $[\term_h]_{\IL(M)}$, for $h = g\circ f$ a morphism of $\Mlaw$. By definition, this is mapped under $\natbar_M$ to $\natbar_M[\term_g]_{\IL(M)} \circ \natbar_M[\term_f]_{\IL(M)} $.

(iii) The functor $\natbar_M$ is an algebraic translation. This is an immediate consequence of the fact that it maps a morphism $\langle [s]_{\IL(M)} \rangle$, for $s$ elementary, to $M\langle s \rangle$, therefore making 
$\natbar_M \big( M_{\IL(M)} (\langle s \rangle)\big) = M\langle s \rangle $. 

The algebraic translations $\nat_M$ and $\natbar_M$ are mutually inverse, which is trivial on their restrictions on objects. The image of a morphism $f \colon \tuple{m}{k} \To (n)$ of $\Mlaw$ under $\natbar_M \circ \nat_M$ is given by
\begin{diagram}[small]
\tuple{m}{k} & \rTo^{\lam (id_{(m_1, \dots, m_k, 0^n)})} & \e{(0)^n}{(m_1, \dots, m_k, 0^n)} & \rTo^{\e{(0)^n}{(\ev_n \circ (f \times (0)^n))}} & (n)
\end{diagram} 
which is equal to $\lam\big( \ev_n \circ (f \times (0)^n) \big)$, which is simply $f$.
In the other direction, we show, by induction on the structure of the term $t$, that for a morphism $\langle [t]_{\IL(M)} \rangle$, 
$(\nat_M \circ \natbar_M)\langle [t]_{\IL(M)} \rangle = \langle [t]_{\IL(M)} \rangle$:
\begin{itemize}
\item[-] For $\mcon{m}{k} \Sep \vcon{x}{n} \vdash x_i$, $(\nat_M \circ \natbar_M)\langle [x_i]_{\IL(M)} \rangle$ is given by the single tuple of the equivalence class of the term 
$ \mcon{m}{k} \Sep \vcon{x}{n} \vdash \term_{M\langle x_i \rangle}$,
which by axiom $(\E 1)$ of $\IL(M)$ is equal to $x_i$.
\item[-] The image of $\big\langle \big[ \mvar{m}_i [t_1, \dots, t_{m_i}] \big]_{\IL(M)} \big\rangle \colon \tuple{m}{k} \To (n)$ under $\nat_M \circ \natbar_M$ is given, by induction on $t_1, \dots, t_{m_i}$, by the single tuple containing the equivalence class of the term
\begin{eqnarray*}
\mcon{m}{k} \Sep \vcon{x}{n} &\vdash& \term_{M \big\langle \mvar{m}_i \big[ \mvar{n}_1 [\vvar{x}], \dots, \mvar{n}_{m_i}[\vvar{x}] \big] \big\rangle} \big\{ \mvar{m}_i \Def (\vvar{y}_i) \term_{M \langle \mvar{m}_i[\vvar{y}_i] \rangle} \big\} \\
& \quad & \hspace{106pt} \big\{ \mvar{n}_j \Def (\vvar{x}) t_j \big\}_{j \in \card{m_i}} \\
& \stackrel{\E 1}{\eq} & \mvar{m}_i \big[ \mvar{n}_1 [\vvar{x}], \dots, \mvar{n}_{m_i}[\vvar{x}] \big]  \big\{ \mvar{m}_i \Def (\vvar{y}_i) \mvar{m}_i[\vvar{y}_i] \big\} \\
& \quad & \hspace{115pt} \big\{ \mvar{n}_j \Def (\vvar{x}) t_j \big\}_{j \in \card{m_i}} \\
& = & \mvar{m}_i [t_1, \dots, t_{m_i}] \quad .
\end{eqnarray*}
\item[-] For $f \colon \tuple{n}{l} \To (j)$ in $\Mlaw$, the image of 
$$ \big\langle \big[ \omega_f \big( (\vvar{y_1})t_1, \dots, (\vvar{y_l}) t_l, s_1, \dots, s_j \big) \big]_{\IL(M)} \big\rangle \colon \tuple{m}{k} \To (n) $$
under $\nat_M \circ \natbar_M$ is the single tuple containing the equivalence class of the term
\begin{eqnarray*}
\mcon{m}{k} \Sep \vcon{x}{n} & \vdash & \term_{\e{(0)^n}{\big(\ev_j \circ (f \times (0)^j)\big)}} \big\{ \mvar{n}_p \Def (\vvar{y}_p) t_p \big\}_{p \in \card{l}} \\
& \quad & \hspace{80pt} \big\{ \mvar{n}'_q \Def (\vvar{x}) s_q \big\}_{q \in \card{j}} \\
& \eq & \term_f \big\{ z_i \Def \mvar{n}'_i [\vcon{x}{n}] \big\}_{i \in \card{j}} \\
& \quad & \hspace{10pt} \big\{ \mvar{n}_p \Def (\vvar{y}_p) t_p \big\}_{p \in \card{l}} \big\{ \mvar{n}'_q \Def (\vvar{x}) s_q \big\}_{q \in \card{j}} \\
& = & \omega_f \big( (\vvar{y_1}) \mvar{n}_1 [\vvar{y_1}], \dots, (\vvar{y_l}) \mvar{n}_l [\vvar{y_l}], \vcon{z}{j} \big)  \\
& \quad & \big\{ z_i \Def \mvar{n}'_i [\vcon{x}{n}] \big\}_{i \in \card{j}} \\
& \quad &  \big\{ \mvar{n}_p \Def (\vvar{y}_p) t_p \big\}_{p \in \card{l}} \big\{ \mvar{n}'_q \Def (\vvar{x}) s_q \big\}_{q \in \card{j}} \\
& = & \omega_f \big( (\vvar{y_1})t_1, \dots, (\vvar{y_l}) t_l, s_1, \dots, s_j \big) \quad .
\end{eqnarray*}
\end{itemize}

Note that we have in fact defined natural isomorphisms
$\nat_{(-)} \colon Id_\SOAT \To \M(\IL(-))$ and $\natbar_{(-)} \colon \M(\IL(-)) \To Id_\SOAT$
with components at a second-order algebraic theory $M \colon \M \To \Mlaw$ given respectively by the algebraic translations $\nat_M$ and $\natbar_M$ defined in the proof above. The proof of this naturality appears in Section ?, where functoriality of $\M (-)$ and $\IL(-)$ is established by defining syntactic translations of internal languages as the image of algebraic translations.



\section*{References}

\bibliographystyle{apalike}
\bibliography{thesisbib}

\begin{thebibliography}{}

\bibitem[Aczel, 1978]{Aczel1978}
Aczel, P. (1978).
\newblock A general church-rosser theorem.
\newblock Typescript.

\bibitem[Aczel, 1980]{Aczel1980}
Aczel, P. (1980).
\newblock Frege structures and the notion of proposition, truth and set.
\newblock In {\em The Kleene Symposium}, pages 31--59.

\bibitem[Adamek et~al., 2009]{AdamekRosickyVitale2009}
Adamek, J., Rosicky, J., and Vitale, E. (2009).
\newblock Algebraic theories: A categorical introduction to general algebra.
\newblock Monograph, available from
  http://www.iti.cs.tu-bs.de/~adamek/adamek.html.

\bibitem[Birkhoff, 1935]{Birkhoff1935}
Birkhoff, G. (1935).
\newblock On the structure of abstract algebras.
\newblock {\em Proceedings of the Cambridge Philosophical Society},
  31:433--454.

\bibitem[Borceux, 1994]{Borceux1994}
Borceux, F. (1994).
\newblock {\em Handbook of Categorical Algebra 1 and 2}, volume~51 of {\em
  Encyclopedia of Mathematics and its Applications}.
\newblock Cambridge University Press.

\bibitem[Burstall, 1969]{Burstall1969}
Burstall, R. (1969).
\newblock Proving properties of programs by structural induction.
\newblock {\em The Computer Journal}, 12(1):41--48.

\bibitem[Church, 1936]{Church1936}
Church, A. (1936).
\newblock An unsolvable problem of elementary number theory.
\newblock {\em American Journal of Mathematics}, 58:354--363.

\bibitem[Church, 1940]{Church1940}
Church, A. (1940).
\newblock A formulation of the simple theory of types.
\newblock {\em Journal of Symbolic Logic}, 5:56--68.

\bibitem[Cohn, 1965]{Cohn1965}
Cohn, P. (1965).
\newblock {\em Universal Algebra}.
\newblock Harper $\&$ Row.

\bibitem[Fiore, 2008]{Fiore2008}
Fiore, M. (2008).
\newblock Second-order and dependently-sorted abstract syntax.
\newblock In {\em LICS 2008}, pages 57--68.

\bibitem[Fiore and Hur, 2007]{FioreHur2007}
Fiore, M. and Hur, C.-K. (2007).
\newblock Equational systems and free constructions.
\newblock In {\em ICALP 2007}.

\bibitem[Fiore and Hur, 2008a]{FioreHur2008a}
Fiore, M. and Hur, C.-K. (2008a).
\newblock On the construction of free algebras for equational systems.
\newblock {\em Theoretical Computer Science}, 410:1704--1729.

\bibitem[Fiore and Hur, 2008b]{FioreHur2008b}
Fiore, M. and Hur, C.-K. (2008b).
\newblock Term equational systems and logics.
\newblock In {\em MFPS XXIV}, pages 171--192.

\bibitem[Fiore and Hur, 2010]{FioreHur2010}
Fiore, M. and Hur, C.-K. (2010).
\newblock Second-order equational logic.
\newblock In {\em CSL 2010}.

\bibitem[Fiore et~al., 1999]{FiorePlotkinTuri1999}
Fiore, M., Plotkin, G., and Turi, D. (1999).
\newblock Abstract syntax with variable binding.
\newblock In {\em LICS 1999}, pages 193--202.

\bibitem[Fujiwara, 1959]{Fujiwara1959}
Fujiwara, T. (1959).
\newblock On mappings between algebraic systems.
\newblock {\em Osaka Mathematical Journal}, 11:153--172.

\bibitem[Fujiwara, 1960]{Fujiwara1960}
Fujiwara, T. (1960).
\newblock On mappings between algebraic systems, ii.
\newblock {\em Osaka Mathematical Journal}, 12:253--268.

\bibitem[Gabbay and Pitts, 2001]{GabbayPitts2001}
Gabbay, M.~J. and Pitts, A.~M. (2001).
\newblock A new approach to abstract syntax with variable binding.
\newblock {\em Formal Aspects of Computing}, 13:341--363.

\bibitem[Goguen et~al., 1978]{GoguenThatcherWagner1978}
Goguen, J., Thatcher, J., and Wagner, E. (1978).
\newblock An initial algebra approach to the specification, correctness and
  implementation of abstract data types.
\newblock In {\em Current Trends in Programming Methodology}, volume~IV, pages
  80--149. Prentice Hall.

\bibitem[Hamana, 2005]{Hamana2005}
Hamana, M. (2005).
\newblock Free $\sigma$-monoids: A higher-order syntax with metavariables.
\newblock In {\em APLAS 2004}, pages 348--363.

\bibitem[Klop, 1980]{Klop1980}
Klop, J. (1980).
\newblock Combinatory reduction systems.
\newblock PhD Thesis, Mathematical Centre Tracts 127, CWI, Amsterdam.

\bibitem[Klop et~al., 1993]{Klop1993}
Klop, J., van Oostrom, V., and van Raamsdonk, F. (1993).
\newblock Combinatory reduction systems: introduction and survey.
\newblock {\em Theoretical Computer Science}, 121:279--308.

\bibitem[Knuth and Bendix, 1970]{KnuthBendix1970}
Knuth, D. and Bendix, P. (1970).
\newblock Simple word problems in universal algebras.
\newblock In {\em Computational Problems in Abstract Algebra}, pages 263--197.

\bibitem[Lawvere, 2004]{Lawvere2004}
Lawvere, F.~W. (2004).
\newblock Functorial semantics of algebraic theories.
\newblock Republished in: Reprints in TAC.

\bibitem[MacLane, 1998]{Maclane1998}
MacLane, S. (1998).
\newblock {\em Categories for the Working Mathematician}, volume~5 of {\em
  Graduate Texts in Mathematics}.
\newblock Springer.

\bibitem[MacLane and Moerdijk, 1992]{MaclaneMoerdijk1992}
MacLane, S. and Moerdijk, I. (1992).
\newblock {\em Sheaves in Geometry and Logic: A First Introduction to Topos
  Theory}.
\newblock Springer Verlag.

\bibitem[Mahmoud, 2011]{Mahmoud2011}
Mahmoud, O. (2011).
\newblock Second-order algebraic theories.
\newblock PhD Dissertation, available from
  https://www.repository.cam.ac.uk/handle/1810/241035.

\bibitem[McCarthy, 1963]{Mccarthy1963}
McCarthy, J. (1963).
\newblock Towards a mathematical science of computation.
\newblock In {\em IFIP Congress 1962}. North-Holland.

\bibitem[Pfenning and Elliott, 1988]{PfenningElliott1988}
Pfenning, F. and Elliott, C. (1988).
\newblock Higher-order abstract syntax.
\newblock In {\em Proceedings of the ACM SIGPLAN PLDI}, pages 199--208.

\bibitem[Plotkin, 1998]{Plotkin1998}
Plotkin, G. (1998).
\newblock Binding algebras: A step from universal algebra to type theory.
\newblock Invited talk at RTA-98.

\bibitem[van Raamsdonk, 2003]{Raamsdonk2003}
van Raamsdonk, F. (2003).
\newblock Higher-order rewriting.
\newblock In {\em Term Rewriting Systems}, volume~55 of {\em Cambridge Tracts
  in Theoretical Computer Science}, pages 588--667.

\bibitem[Vidal and Tur, 2008]{VidalTur2008}
Vidal, J.~C. and Tur, J.~S. (2008).
\newblock On the morphisms and transformations of tsuyoshi fujiwara.
\newblock Unpublished notes, available from
  http://www.uv.es/jkliment/Documentos/.

\end{thebibliography}







\end{document}